\documentclass[12pt]{article}
\usepackage[english]{babel}
\usepackage{amsmath,amsthm,amsfonts,amssymb,epsfig,xcolor,hyperref}
\usepackage[left=1in,top=1in,right=1in]{geometry}

\newtheorem{thm}{Theorem}[section]
\newtheorem{lem}[thm]{Lemma}
\newtheorem{prop}[thm]{Proposition}

\newtheorem{df}[thm]{Definition}
\newtheorem{rem}[thm]{Remark}

\numberwithin{equation}{section}

\newcommand{\E}{\mathbb{E}}

\newcommand{\prob}{\mathbb{P}}
\renewcommand{\P}{\mathbb{P}}

\begin{document}

\title{Zero-temperature Glauber dynamics on the $3$-regular tree and the median process}
\date{\today}
\author{\Large Michael Damron \thanks{The research of MD is supported by an NSF CAREER grant.}\\ \small Georgia Tech \and \Large Arnab Sen \thanks{The research of AS is partially supported by NSF DMS 1406247.}\\ \small University of Minnesota}
\maketitle

\begin{abstract}
In zero-temperature Glauber dynamics, vertices of a graph are given i.i.d.~initial spins $\sigma_x(0)$ from $\{-1,+1\}$ with $\mathbb{P}_p(\sigma_x(0) = +1)=p$, and they update their spins at the arrival times of i.i.d.~Poisson processes to agree with a majority of their neighbors. We study this process on the 3-regular tree $\mathbb{T}_3$, where it is known that the critical threshold $p_c$, below which $\mathbb{P}_p$-a.s.~all spins fixate to $-1$, is strictly less than $1/2$. Defining $\theta(p)$ to be the $\mathbb{P}_p$-probability that a vertex fixates to $+1$, we show that $\theta$ is a continuous function on $[0,1]$, so that, in particular, $\theta(p_c)=0$. To do this, we introduce a new continuous-spin process we call the median process, which gives a coupling of all the measures $\mathbb{P}_p$. Along the way, we study the time-infinity agreement clusters of the median process, show that they are a.s.~finite, and deduce that all continuous spins flip finitely often. In the second half of the paper, we show a correlation decay statement for the discrete spins under $\mathbb{P}_p$ for a.e.~value of $p$. The proof relies on finiteness of a vertex's ``trace'' in the median process to derive a stability of discrete spins under finite resampling. Last, we use our methods to answer a question of C. Howard (2001) on the emergence of spin chains in $\mathbb{T}_3$ in finite time.
\end{abstract}

\maketitle

%\section{To do}
%\begin{enumerate}
%\item Introduction, motivation.
%\end{enumerate}

\section{Introduction}

\subsection{The model}

%Consider the graph $\mathbb{T}_3$, the infinite $3$-regular tree. The vertex set will be denoted by $\mathcal{V} = V(\mathbb{T}_3)$ and the edge set $\mathcal{E} = E(\mathbb{T}_3)$. Write $o$ for a distinguished vertex (the root) of $\mathbb{T}_3$. 

We study the majority vote model, known as zero-temperature Ising Glauber dynamics, on $\mathbb{T}_3$, the infinite $3$-regular tree with vertex set $\mathcal{V}$ and edge set $\mathcal{E}$. This is a continuous-time Markov process whose state space is $\{-1,+1\}^{\mathcal{V}}$, with vertices updating their values at times according to rate-one Poisson clocks to agree with a majority of their neighbors. We take an initial spin configuration $\sigma(0) = (\sigma_x(0))_{x \in \mathcal{V}} \in \{-1,+1\}^{\mathcal{V}}$ distributed according to the i.i.d.~Bernoulli product measure $\mu_p$, $p \in [0,1]$, where
\[
\mu_p(\sigma_x(0) = +1) = p = 1 - \mu_p(\sigma_x(0) = -1).
\]
The configuration $\sigma(t)$ evolves as $t$ increases according to the zero-temperature limit of Glauber dynamics. To describe this, define the energy (or local cost function) of a vertex $x$ at time $t$ as
\[
e_x(t) = - \sum_{y \in \partial x} \sigma_x(t) \sigma_y(t),
\]
where $\partial x$ is the set of three neighbors of $x$ in $\mathbb{T}_3$. 
Note that $e_x(t) $ is  the number of neighbors of $x$ that disagree with $x$ minus the number of neighbors that agree with $x$ at time $t$. Each vertex has an exponential clock of rate 1 and clocks at different vertices are independent of each other. When a vertex's clock rings, it makes an update according to the rules
\begin{equation}\label{eq: energy_rules}
\sigma_x(t) = \begin{cases}
- \sigma_x(t^-) & \quad \text{if } e_x(t^-) > 0 \\
\sigma_x(t^-) & \quad \text{if } e_x(t^-) < 0
\end{cases}.
\end{equation}
Hence each spin flips with probability 1 (or 0) when it disagrees (or agrees) with the majority of its neighbors. Write $\mathbb{P}_p$ for the joint distribution of $(\sigma(0),\omega)$, the initial spins and the dynamics realizations (Poisson clocks).

We are interested in the limiting behavior of spins $\sigma_x(t)$ as $t \to \infty$, and we first note that this limit exists a.s., due to a general theorem of Nanda-Newman-Stein (see the remark after Theorem~2 in \cite{NNS}). Their result implies that on general odd degree graphs with enough symmetries (including $\mathbb{T}_3$), for any $p\in [0,1]$, $\mathbb{P}_p$-a.s. every spin flips only finitely often. So we can define the limit
\begin{equation}\label{eq: discrete_limit}
\sigma_x(\infty) = \lim_{t \to \infty} \sigma_x(t) \quad\text{for } x \in \mathcal{V}.
\end{equation}
Even basic properties of the limiting configuration $(\sigma_x(\infty))$ are not established, and this is in part due to the fact that the spins at time infinity are highly correlated. Although $(\sigma_x(t))$ for finite $t$ has a range of dependence that decays exponentially (Lemma~\ref{lem: chronological}), for arbitrarily large times, information can propagate arbitrarily large distances, and this decay is lost. For this reason, we cannot apply results that are designed for nearly independent processes, and must instead rely on tools from invariant percolation. Many of our methods apply to some degree to $\mathbb{T}_d$ for $d \geq 3$ and odd, and therefore we believe our perspective will be useful to analyze similar questions on these trees. There are, however, notable exceptions: seen in Lemmas~\ref{lem: no_degree_one}, \ref{lem: agree_with_neighbor}, and \ref{l:NNS_finite_osc}. Each of these can be traced back to defining a suitable energy function for our continuous-spin analogue, the median process. The natural energy $E_x(t)$ defined in \eqref{eq: natural_energy} cannot increase by a flip of the continuous spin at $x$ when $x$ has degree 3, but this is false when the degree is at least 5.

\subsection{Main results}

Our results are of two types: those for the discrete spin model defined above, and those for a continuous spin variant we call the median process.

\subsubsection{Results for the discrete model}
Due to \eqref{eq: discrete_limit}, we can define
\begin{equation}\label{eq: theta_def}
\theta(p) = \mathbb{P}_p(\sigma_o(\infty) = +1),
\end{equation}
where $o$ (``the root'') is a distinguished vertex of $\mathbb{T}_3$, and therefore
\begin{equation}\label{eq: p_c_def}
p_c = \sup \left\{p\in [0,1] : \theta(p) = 0\right\}.
\end{equation}
By a straightforward coupling of initial configurations and clocks, it can be seen by attractiveness (see the paragraph below \eqref{eq: infinite_chronological}) that $\theta(p)$ is nondecreasing in $p$, and so $\theta(p)=0$ for $p \in [0,p_c)$. By symmetry, $\theta(p) = 1$ for $p \in (1-p_c,1]$. It has been shown \cite[Theorem~4]{HowardTree} that $(2-\sqrt{3})/4 \leq p_c < 1/2$ for $\mathbb{T}_3$.

Our first result states that $\theta$ is a continuous function.

\begin{thm}\label{thm: continuity}
The function $\theta$ is absolutely continuous in $p$. In particular, $\theta(p_c)=0$.
\end{thm}
\noindent
The proof of Theorem~\ref{thm: continuity} will be given in Section~\ref{sec: continuity_proof}.

The next result is a type of decay of correlations for the the limiting spins $(\sigma_x(\infty))$. Write $\text{dist}$ for the graph distance on $\mathbb{T}_3$ and for $r\geq 0$, define the sigma-algebras
\begin{align*}
\Sigma_{\leq r} &= \sigma\left( \sigma_x(\infty) : \text{dist}(o,x) \leq r\right) \\
\Sigma_{\geq r} &= \sigma\left( \sigma_x(\infty) : \text{dist}(o,x) \geq r\right).
\end{align*}
For $r,R$ with $0 \leq r \leq R < \infty$, and $p \in [0,1]$, define the strong mixing coefficient
\[
\alpha_{r,R}(p) = \sup\left\{|\mathbb{P}_p(A \cap B) - \mathbb{P}_p(A) \mathbb{P}_p(B)| : A \in \Sigma_{\leq r}, B \in \Sigma_{\geq R}\right\}.
\]
The following theorem, whose proof appears in Section~\ref{sec: correlation_decay}, states that the coefficient converges to zero as $R \to \infty$ for fixed $r$. 
%It would be interesting to find the rate of decay of $\alpha_{r,R}(p)$ to zero as $R$ grows for fixed $r$, and, in particular, to determine whether it is exponential.
\begin{thm}\label{thm: correlation_decay}
For Lebesgue-a.e. $p \in [0,1]$ one has for every fixed $r \geq 0$,
\[
\lim_{R \to \infty} \alpha_{r,R}(p) = 0.
\]
\end{thm}
\noindent
Using the fact that the configuration $(\sigma_x(\infty) : x \in \mathbb{T}_3)$ is a factor of i.i.d.\  process, it follows that the two-point correlation between $\sigma_x(\infty)$ and $\sigma_y(\infty)$ decays exponentially in the distance between vertices $x$ and $y$ for every $p$  (see \cite[Theorem 3.1]{BSV_corr_decay}). Theorem~\ref{thm: correlation_decay} involves events that depend on finitely many spins near the root, but infinitely many spins away from the root. We are only able to show the equation of this theorem for a.e.\ $p$ because it is obtained from Theorem~\ref{thm: finite_trace} about the median process after applying a certain projection (in the proof of Theorem~\ref{thm: resample_spin}) and Fubini's theorem.

To derive Theorem~\ref{thm: correlation_decay}, we prove a type of stability for the process $\sigma(\cdot)$ under finite perturbations. Following the notation in Section~\ref{sec: correlation_decay}, we define the discrete evolutions $\sigma^+(\cdot)$ and $\sigma^-(\cdot)$ using the same Poisson clocks $\omega$ as the original process, but starting from initial conditions $\sigma^+(0)$ and $\sigma^-(0)$. The first, $\sigma^+(0)$, equals $\sigma(0)$ at vertices apart from $o$, and equals $+1$ at $o$, whereas $\sigma^-(0)$ equals $\sigma(0)$ at vertices apart from $o$, and equals $-1$ at $o$. Their symmetric difference is defined
\[
\sigma^+(\cdot) \Delta \sigma^-(\cdot) = \{y \in \mathcal{V} : \sigma_y^+(t) \neq \sigma_y^-(t) \text{ for some } t \geq 0\}.
\]
The following theorem is restated as Lemma~\ref{lem: resample_spin} and is proved in Section~\ref{sec: correlation_decay}.
\begin{thm}\label{thm: resample_spin}
For Lebesgue-a.e. $p \in [0,1]$, one has
\[
\mathbb{P}_p\left( \#\sigma^+(\cdot) \Delta \sigma^-(\cdot) < \infty\right) = 1.
\]
\end{thm}

Our last result on the discrete model answers a question of Howard. In his notation, we say that a spin chain is a doubly-infinite path of vertices all of whose spins agree. Let $\mathbb{T}^{+\text{chains}}(t)$ (respectively $\mathbb{T}^{-\text{chains}}(t)$) denote the subgraph of $\mathbb{T}_3$ generated by those vertices belonging to $+1$ (respectively $-1$) spin chains at time $t$ and put $\mathbb{T}^{\text{chains}}(t) = \mathbb{T}^{+\text{chains}}(t) \cup \mathbb{T}^{-\text{chains}}(t)$. The graphs $\mathbb{T}^{+\text{chains}}, \mathbb{T}^{-\text{chains}},$ and $\mathbb{T}^{\text{chains}}$ are defined similarly, using $t=\infty$. Note that spin chains are stable under the majority updates, and for any $x \in \mathcal{V}$ one has $\sigma_x(\infty) = +1$ if and only if $x \in \mathbb{T}^{+\text{chains}}$ (similarly for $-1$ chains). On \cite[p.~739]{HowardTree}, Howard states ``A central question is whether every vertex joins a spin chain in \emph{finite} time, or, equivalently, that, as a vertex set, $\mathbb{T}^{\text{chains}}\uparrow \mathbb{T}_3$ as $t \to \infty$.'' The next theorem shows that the answer to this question is yes.
\begin{thm}\label{thm: spin_chains}
Let $p \in [0,1]$. Then
\[
\mathbb{P}_p(o \in \mathbb{T}^{\text{chains}}(t) \text{ for some } t < \infty) = 1.
\]
\end{thm}

To prove Theorem~\ref{thm: spin_chains}, we show the following qualitative description of the limiting configuration $\sigma(\infty)$. We define the subgraph $\mathbb{G}_+$ of $\mathbb{T}_3$ generated by those vertices $x$ with $\sigma_x(\infty) = +1$. A component of $\mathbb{G}_+$ has infinitely many ends if for any $M>0$, there is a finite subgraph of $\mathbb{G}_+$ whose removal from $\mathbb{G}_+$ splits it into at least $M$ components.
\begin{thm}\label{thm: infinitely_many_ends}
For any $p \in [0,1]$ with $\theta(p)>0$, 
\[
\mathbb{P}_p(\text{each component of }\mathbb{G}_+ \text{ has infinitely many ends})=1.
\]
\end{thm}

\subsubsection{Results for the median process}

The main results on the discrete spin process come from analyzing a continuous spin version of the model which we call the median process. It will provide a coupling of discrete processes for different values of $p$ on the same probability space with some probability measure $\mathbb{P}$. For this, we let $(U_x)_{x \in \mathcal{V}} \in [0,1]^{\mathcal{V}}$ be a family of i.i.d. uniform $[0,1]$ random variables, one assigned to each vertex of $\mathbb{T}_3$. Consider the following dynamics, starting from the initial configuration $U(0) = (U_x(0))$, where $U_x(0):=U_x$, producing a Markov process $U(\cdot)$ on the space $[0,1]^{\mathcal{V}}$. Each vertex has an exponential clock with rate one and clocks at different sites are independent of each other. If a vertex's clock rings at time $t$, then it assumes the median of its neighbors' values:
\[
U_x(t) = \text{median}\left\{ U_y(t^-) :  y \in \partial x \right\}.
\]
(One can check that this process is well-defined by applying \cite[Theorem~3.9]{liggett} with the choices $W=[0,1]$, $S = \mathbb{T}_3$, $X = [0,1]^S$, and $c_T$ defined as $c_T(\eta,\text{d}\zeta) = 0$ if $\#T \geq 2$ and $c_T(\eta,\text{d}\zeta) = \delta_{\{\text{median}(\eta(y_i), y_i \in \partial x)\}}$ if $T = \{x\}$.) This ``continuous  spin'' model is related to the original discrete spin model by projection, and we show this in Lemma~\ref{lem: projection}. Namely, for a given $p \in [0,1]$, the set of vertices $x$ with $U_x(t) \leq p$ evolves as $t$ grows as the set of vertices with $+1$ spins in the discrete spin model distributed under the measure $\mathbb{P}_p$.

In Lemma~\ref{lem: continuous_limit}, we show that each continuous spin has a limit as $t\to\infty$: for $x \in \mathcal{V}$, we can a.s.~define
\[
U_x(\infty) := \lim_{t \to \infty} U_x(t).
\]
Unlike in the discrete process, the fact that the continuous spins have limits does not immediately imply that they are eventually constant in time. Our next main result states that this is, however, indeed true. To state it, we say that the vertex $x$ has a flip at time $t$ if $U_x(t^-) \neq U_x(t)$. 
\begin{thm}\label{thm: continuous_theorem}
Almost surely, each $x$ flips only finitely often: for any $x \in \mathcal{V}$,
\[
\mathbb{P}\left(U_x(t^-) \neq U_x(t) \text{ for infinitely many }t\right)=0.
\]
\end{thm}
Theorem~\ref{thm: continuous_theorem} is a direct consequence of Proposition~\ref{prop: infinite_equivalence} and Theorem~\ref{thm: finite_agreement_clusters} below. Furthermore, it implies Theorem~\ref{thm: continuity} (see Section~\ref{sec: conditional_proof}).

Our next results are about the structure of agreement and disagreement clusters of the limiting state of the median process. 
\begin{df}\label{def: agreement}
The agreement graph is the graph whose vertex set is $\mathcal{V}$ and whose edge set consists of those edges $\{x,y\} \in \mathcal{E}$ such that $U_x(\infty) = U_y(\infty)$. For $x \in \mathcal{V}$, the agreement cluster of $x$, written $\mathcal{A}_x$, is the component of the agreement graph containing $x$.
\end{df}

%The agreement graph (see Definition~\ref{def: agreement}) is the subgraph of $\mathbb{T}_3$ whose vertex set is $\mathcal{V}$ and whose edge set is equal to the set of agreement edges --- those edges $\{x,y\} \in \mathcal{E}$ such that $U_x(\infty) = U_y(\infty)$. 

One can to argue that on $\mathbb{T}_3$, vertices $w,z$ in distinct components of the agreement graph have $U_w(\infty) \neq U_z(\infty)$. A main result used to derive our discrete theorems is the following, which will be shown in Section~\ref{sec: agreement_finite}.
\begin{thm}\label{thm: finite_agreement_clusters}
Almost surely,  all agreement clusters are finite.
\end{thm}

Correspondingly, we may define disagreement clusters. The disagreement graph is the subgraph of $\mathbb{T}_3$ whose vertex set is $\mathcal{V}$ and whose edge set is equal to the set of disagreement edges --- those edges $\{x,y\}\in \mathcal{E}$ such that $U_x(\infty) \neq U_y(\infty)$. The main result on the disagreement graph is that a.s., all of its components are finite.
\begin{thm}\label{thm: disagreement_clusters}
Almost surely, all components of the disagreement graph are finite. Furthermore, for any component $\mathcal{C}$, there are two vertices $w,z \in \mathcal{C}$ such that $\mathcal{C}$ is equal to the path in $\mathbb{T}_3$ between $w$ and $z$.
\end{thm}
\noindent
Theorem~\ref{thm: disagreement_clusters} is proved in Section~\ref{sec: disagreement_clusters}.

The last result concerns the trace of a vertex. In Definition~\ref{def: trace}, we define the trace of $x\in \mathcal{V}$ as $\text{Tr}(x)$, the set of $y \in \mathcal{V}$ such that $U_y(t) = U_x(0)$ for some $t$. We restate the following theorem as Proposition~\ref{prop: finite_trace}, and prove it in Section~\ref{sec: correlation_decay} as a tool to derive Theorem~\ref{thm: correlation_decay}.
\begin{thm}\label{thm: finite_trace}
Almost surely,  for each $x \in \mathcal{V}$, one has $\#\text{Tr}(x)<\infty$.
\end{thm}

%{\color{red} Remark that spin to spin correlations decay easily. Also talk about how finite change to initial spins has finite effect. (Call this perturbation result for discrete model.)} 

%Our last result on the median process is a structure theorem for projected clusters that holds simultaneously for all $p$. Recall that a subgraph $\mathbb{G}$ of $\mathbb{T}_3$ has at least $k$ ends if there is a finite set of vertices $S$ such that the subgraph obtained by removing the vertices of $S$ from $\mathbb{G}$ (and their incident edges) has at least $k$ infinite components. Write $\mathbb{G}_p$ for the subgraph of $\mathbb{T}_3$ induced by the vertices $x$ with $U_\infty(x) \leq p$. {\color{red} this graph was defined later. consolidate.}
%\begin{thm}\label{thm: simultaneous_ends}
%For the median process on $\mathbb{T}_3$, a.s. there is no $p \in [0,1]$ such that a component of $\mathbb{G}_p$ has finitely many ends.
%\end{thm}

%Theorems~\ref{thm: disagreement_clusters} and \ref{thm: simultaneous_ends} are proved in Section~\ref{sec: structure_of_clusters}.

\subsection{Relation to other work}

The majority vote model we study here is typically considered on $\mathbb{Z}^d$ under the name ``zero-temperature Ising Glauber dynamics,'' since the update rule is the zero-temperature limit of the one for Ising Glauber dynamics. Because every vertex of $\mathbb{Z}^d$ has even degree, there are ties in the majority rule, and one needs a tie-breaking procedure. Usually a fair coin is used (update to $+1$ with probability $1/2$ if a vertex has an equal number of $+1$ and $-1$ neighbors). In this context, the main questions involve fixation (whether each spin flips finitely many times) and the value of the consensus threshold, defined as
\[
p_c = \sup\{ p \in [0,1] : \mathbb{P}_p(\sigma_o(\infty) \text{ exists and equals } -1) = 1\}.
\]

On $\mathbb{Z}$, the majority vote process is equivalent to the nearest-neighbor voter model (spins update to $+1$ with probability proportional to the number of neighbors with spin $+1$). In \cite{arratia}, it is shown that for this voter model, and any $p \in (0,1)$, each spin flips infinitely often, and therefore there is no fixation. This implies $p_c=0$. For $d \geq 2$, Fontes-Schonman-Sidoravicius \cite{FSS} proved that $p_c \in (0,1/2]$ by using a multiscale analysis. The method of \cite{FSS} does not give a quantitative lower bound on $p_c$, whose value is predicted by a folklore conjecture to be $1/2$ for all $d \neq 1$. The closest result to this conjecture is due to Morris \cite{Morris}, who showed that, as a function of $d$, $p_c$ approaches $1/2$ as $d \to \infty$. In another direction, very little is known about fixation for general $d$. Nanda-Newman-Stein \cite{NNS} used an invariance argument to prove that for $p=1/2$ and $d=2$, every spin flips infinitely often, and this is believed to be true for $p=1/2$ and $d$ sufficiently small. (See \cite{SKR, S} for numerical evidence that fixation holds for $d$ large.) The method of \cite{NNS} shows further that for odd-degree graphs with sufficiently many symmetries, one has fixation, and for even-degree graphs, each spin has finitely many ``energy lowering'' flips. In \cite{CNS}, Camia-Newman-Sidoravicius studied the model on the two-dimensional hexagonal lattice (which has degree three), and showed that the process has remarkably simple properties: each spin flips at most 8 times, and fixation occurs due to local structures which are stable under the dynamics. This latter fact can be used to show that $\theta$ is continuous. In the case of $\mathbb{T}_3$ that we study, fixation is governed instead by doubly-infinite paths and does not obviously imply continuity of $\theta$.

Even for trees, the results are scarce. For $\mathbb{T}_d$ with $d$ odd, the Nanda-Newman-Stein result implies that fixation holds a.s.~for each vertex. For even-degree trees, the only result is for $p=1/2$, where it is shown that a.s., there exist vertices that do not fixate (see for example \cite[Thm.~39]{Tessler}). The question of the value of the consensus threshold is also wide open in general. There is some evidence (based on work on the synchronous analogue of the majority vote model by Kanoria-Montanari \cite{KM11}) that $p_c < 1/2$ for odd $d \geq 3$ and $p_c=1/2$ for even $d\geq 4$. The only verified case is $d=3$ by Howard, who showed that $p_c<1/2$ using a branching process argument and short-time analysis. It follows from the work of Caputo-Martinelli \cite{caputo} that $p_c(\mathbb{T}_d) \to 1/2$ as $d \to \infty$.

%{\color{green} Things to mention:
%
%TREES:
%
%\begin{enumerate}
%\item Howard for 3-regular trees. ($p_c<1/2$ short time analysis)
%\item 4- regular tree at $p=1/2$, some vertices flip i.o.
%\item Kanoria-Montanari for $p_c<1/2$ synchronous
%\item Mention Caputo ($p_c \to 1/2$ as $d\to\infty$)
%\end{enumerate}
%
%$\mathbb{Z}^d$:
%\begin{enumerate}
%\item Arratia for $d=1$
%\item NNS shows for $d=2$, $p=1/2$, each spin flips infinitely often
%\item FFS that $p_c \in (0,1)$.
%\item Camia-Newman-Sidoravicius on hexagonal lattice (showed local structure and so $\theta$ is continuous, etc.)
%\item Rob Morris result (same but for $\mathbb{Z}^d$)
%\end{enumerate}
%}

\subsection{Mass transport, invariant percolation, and other tools}\label{sec: tools}

In this section, we present some tools that will be used in the remainder of the paper. The first tool we need is the mass transport principle. This is an important part of nearly all of our proofs, and the reader can see introductions in \cite{Haggstrom2} and \cite[Section~3]{Haggstrom}. Let $\mathbb{P}$ be a probability measure on $[0,1]^{V(\mathbb{T}_d)}$, where $V(\mathbb{T}_d)$ is the vertex set of the infinite $d$-regular tree, which is invariant under each tree-automorphism $\Theta$. Next, suppose that a random variable $m(x,y) \geq 0$ (mass-transport rule) for each pair of vertices $x,y$ in $\mathbb{T}_d$ has been defined such that for all tree-automorphisms $\Theta$, one has
\begin{equation}\label{eq: mass_invariance}
\mathbb{E}m(x,y) = \mathbb{E}m(\Theta(x),\Theta(y)).
\end{equation}
Then the mass transport principle states that the expected mass entering a vertex equals the expected mass exiting a vertex; that is,
\[
\mathbb{E}\sum_{y\in \mathcal{V}} m(x,y) = \mathbb{E}\sum_{x\in \mathcal{V}} m(x,y).
\]
The mass transport principle is valid in our context because the majority dynamics and the median process are obviously invariant under tree-automorphisms.

Often we will apply the mass transport principle using invariant orderings to break ties in the mass transport rule. For example, suppose we run the median process until time $t=1$ and define a mass transport as
\[
m(x,y) = \begin{cases}
1 & \quad \text{if } y \text{ is the closest vertex to } x \text{ with } U_y(1) \leq 1/2 \\
0 & \quad \text{otherwise}
\end{cases}.
\]
It may be that $U_x(1) > 1/2$, while all neighbors of $x$ have spin values $\leq 1/2$. In this case, to resolve the ambiguity in ``closest,'' we apply what we will call an ``invariant tie-breaking rule.'' Let $(\xi_v)$ be a family of uniform $[0,1]$ i.i.d.~random variables (independent of all the other variables), one assigned to each vertex of $\mathbb{T}_3$. To break ties, we can write $Y_x$ for the subset of vertices of $\{y \in \mathcal{V} : U_y(1) \leq 1/2\}$ with minimal distance to $x$ and define
\[
m(x,y) = \begin{cases}
1 & \quad \text{if } \xi_y \text{ is maximal among } \{\xi_z : z \in Y_x\} \\
0 & \quad \text{otherwise}
\end{cases}.
\]
This tie-breaking rule preserves invariance of $m(x,y)$ under automorphisms in \eqref{eq: mass_invariance}.

In our proofs we will need results about invariant percolation on $\mathbb{T}_3$. These will be proved in Section~\ref{sec: percolation}, but here we will state the most relevant one for future use. Consider the measure space $\{0,1\}^{\mathcal{E}}$ with the product sigma-algebra (generated by the cylinders). An invariant (bond) percolation process is a random element $\eta$ of this space (a measurable map from a probability space to $\{0,1\}^{\mathcal{E}}$) whose distribution is invariant under all graph automorphisms of $\mathbb{T}_3$. For a given element $\tau \in \{0,1\}^{\mathcal{E}}$, we say that an edge $e$ is open in $\tau$ if $\tau(e) = 1$ and closed otherwise. We say that a path in $\mathbb{T}_3$ is open in $\tau$ if all of its edges are open. For a given $\tau$, the open cluster of a vertex $v \in \mathcal{V}$, written $C_\tau(v)$, is the subgraph of $\mathbb{T}_3$ whose vertex set $V(C_\tau(v))$ is the set of $w \in \mathcal{V}$ such that there is an open path starting from $v$ and ending at $w$, and whose edge set $E(C_\tau(v))$ is the set of open edges whose endpoints are in the vertex set of $C_\tau(v)$. Note that $v \in V(C_\tau(v))$, but the edge set $E(C_\tau(v))$ can be empty. If $\tau_1,\tau_2$ are elements of $\{0,1\}^{\mathcal{E}}$, we write $\tau_1 \leq \tau_2$ if for all $e \in \mathcal{E}$, one has $\tau_1(e) \leq \tau_2(e)$.

For a given invariant bond percolation process $\eta$ and a vertex $v \in \mathcal{V}$, let $A_\eta(v)$ be the event that $C_\eta(v)$ has at least three ends. Recall this means that there is a finite subgraph of $C_\eta(v)$ such that the graph obtained from $C_\eta(v)$ by removing this subgraph has at least three infinite components. Also, let $D_\eta(v)$ be the event that $v$ is in a self-avoiding doubly-infinite $\eta$-open path. 

\begin{lem}\label{cor: percolation}
Let $\eta$ be an invariant bond percolation process on $\mathbb{T}_3$ and let $(\eta_n)$ be a sequence of invariant bond percolation processes on $\mathbb{T}_3$ such that 
\begin{enumerate}
\item $\eta_n \leq \eta_{n+1} \leq \eta$ a.s. for all $n$, and
\item for any $e \in \mathcal{E}$, $\lim_{n\to\infty} \mathbb{P}(\eta_n(e) = 1) = \mathbb{P}(\eta(e)=1)$.
\end{enumerate}
Then if $\mathbb{P}(A_\eta(o),D_\eta(o))>0$, one has
\begin{equation}\label{eq: to_prove_doubly_infinite}
\lim_{n\to\infty} \mathbb{P}(D_{\eta_n}(o) \mid A_\eta(o), D_\eta(o)) = 1.
\end{equation}
\end{lem}
The proof will be given in Section~\ref{sec: percolation}, and will follow from a more general percolation result, Proposition~\ref{prop: percolation}.

The next tool relates to ``finite speed of propagation,'' and can be used to ensure approximate independence between well-separated regions for finite times. In both the discrete and continuous models, for the spins in a set of vertices $V$ to influence the spins in another set of vertices $W$ before time $T$, there must exist a chronological path starting from $V$ and ending in $W$. This is a path with (possibly repeating) vertices $x_0, x_1, \ldots, x_k$ satisfying $x_0 \in V$ and $x_k \in W$, and such that the Poisson clocks at the $x_i$'s ring in succession during the interval $[0,T]$. That is, the clock at $x_0$ rings at a time $t_0 \in [0,T]$, the clock at $x_1$ rings at a time $t_1 > t_0$ in $[0,T]$, and so on. The following lemma bounds the probability that there is a long chronological path starting at $o$ for the interval $[0,T]$.
\begin{lem}\label{lem: chronological}
Given $T\geq 0$,
\begin{align*}
\mathbb{P}(\exists \text{ chronological path for } [0,T] \text{ with } \geq k \text{ many vertices starting from or }&\text{ending at }o) \\
&\leq \frac{5e^{4T}}{4} \left( \frac{4}{5}\right)^k.
\end{align*}
\end{lem}
The proof of this lemma will appear in the appendix.
%If $\Gamma$ is a deterministic path starting from $o$ with $\ell$ many vertices, the time it takes for successive clock rings to occur along $\Gamma$ is at least $\sum_{i=1}^\ell \tau_i$, where the $\tau_i$'s are exponential mean one random variables. There are $3^{\ell-1}$ many paths starting from or ending at $o$ with $\ell$ many vertices, so by a union bound and the Markov inequality,
%\begin{align*}
%&\mathbb{P}(\exists \text{ chronological path with } \ell \text{ many vertices starting from or ending at }o \text{ for } [0,T]) \\
%\leq~&3^{\ell-1} \mathbb{P}\left( \sum_{i=1}^\ell \tau_i \leq T\right) \\
%=~& 3^{\ell-1} \mathbb{P}\left( \exp\left( - 3 \sum_{i=1}^\ell \tau_i \right) \geq e^{-3T}\right) \\
%\leq~&3^{\ell-1} \frac{\mathbb{E}\exp\left( - 3 \sum_{i=1}^\ell \tau_i\right)}{e^{-3T}} \\
%=~& \frac{e^{3T}}{3} \left( \frac{3}{4} \right)^\ell.
%\end{align*}
%Summing this bound over $\ell \geq k$ gives the statement of the lemma.
A straightforward consequence of Lemma~\ref{lem: chronological} is that for any $T\geq 0$, 
\begin{equation}\label{eq: infinite_chronological}
\mathbb{P}\left( \exists \text{ infinite chronological path for } [0,T]\right) = 0.
\end{equation}
%To prove this, one applies Lemma~\ref{lem: chronological} with a vertex $x$ in place of $o$ and uses a union bound over all $x$.

Last, we mention attractiveness, which is a type of monotonicity-preservation property for the discrete dynamics. Consider two sets of initial conditions, $(\sigma^{(1)}(0),\omega)$ and $(\sigma^{(2)}(0),\omega)$, where $\sigma^{(i)}(0) \in \{-1,+1\}^{\mathcal{V}}$ for $i=1,2$, and $\omega$ is one realization of Poisson clocks for which the discrete dynamics are defined. Suppose that $\sigma^{(1)}_x(0) \leq \sigma^{(2)}_x(0)$ for all $x \in \mathcal{V}$. Then by attractiveness of the model, we mean the fact that under these conditions,
\[
\sigma^{(1)}_x(t) \leq \sigma^{(2)}_x(t) \text{ for all }x \in \mathcal{V} \text{ and } t \geq 0.
\]
In other words, if two initial configurations are ordered, and if both processes using these configurations use the same Poisson clocks, then the configurations will remain ordered for all time (see \cite[p.~192]{liggett}).

\section{Proof of continuity}\label{sec: continuity_proof}

\subsection{Coupling between the discrete model and the median process}

We begin by showing that the median process and the discrete spin model are related via a projection. The results of this subsection are valid for more general odd-degree graphs, like odd-degree regular trees.

\begin{lem}\label{lem: projection}
For $p \in [0,1]$, let $\pi_p : [0,1]^{\mathcal{V}} \to \{-1,+1\}^{\mathcal{V}}$ be the projection map given by
\[
(\pi_p(\hat U))_x = 2\mathbf{1}_{\{\hat U_x \leq p\}} -1 = \begin{cases}
+1 & \quad \text{if } \hat U_x \leq p \\
-1 & \quad \text{otherwise},
\end{cases}
\]
where $\hat U = (\hat U_x)$ is an arbitrary element of $[0,1]^{\mathcal{V}}$. For a fixed $p$, the process $(\pi_p(U(t)))_{t \geq 0}$ has the same distribution (on $\{-1,+1\}^{\mathcal{V} \times [0,\infty)}$) as $(\sigma(t))_{t \geq 0}$ does under the measure $\mathbb{P}_p$. In particular, the configuration $\lim_{t \to \infty} \pi_p(U(t))$, given by 
\[
\left( \lim_{t \to \infty} \pi_p(U(t)) \right)_x = \begin{cases}
+1 & \quad \text{if } U_x(t) \leq p \text{ for all large }t \\
-1 & \quad \text{otherwise},
\end{cases}
\]
exists and has the same distribution as $\sigma(\infty)$ does under $\mathbb{P}_p$.
%The same statement holds if $\leq p$ is replaced by $<p$.
\end{lem}
\begin{proof}
%We will argue for the statement with $\leq p$; the proof for $<p$ is similar. 
Because $\lim_{t\to\infty} \pi_p(U(t))$ is a measurable function of the process $(\pi_p(U(t)))_{t \geq 0}$, and $\sigma(\infty)$ is a measurable function of the process $\sigma(\cdot)$, the fact that these configurations have the same distribution is a consequence of the fact that the processes have the same distribution. Therefore, we prove the equality in distribution for the processes.

If we are given an initial condition $(\sigma(0),\omega)$ of a discrete spin configuration and a realization of Poisson clocks $\omega$, then the majority vote dynamics are deterministic. The same is true if we are given an initial condition $(U(0), \omega)$ of continuous spins and Poisson clocks and run the median dynamics. To make this explicit, let $\Phi$ be the (measurable) map which sends $(\sigma(0), \omega)$ to the element $\sigma(\cdot) \in \{-1,+1\}^{\mathcal{V} \times [0,\infty)}$ and let $\Psi$ be the (measurable) map sending $(U(0),\omega)$ to the element $U(\cdot) \in [0,1]^{\mathcal{V} \times [0,\infty)}$. Then, extending the projection $\pi_p$ to $[0,1]^{\mathcal{V} \times [0,\infty)}$ as $\pi_p^*$ by $\pi_p^*(U(\cdot)) = (\pi_p(U(t)))_{t \geq 0}$, we claim that
\begin{equation}\label{eq: commutative}
\Phi(\pi_p(U(0)), \omega) = \pi_p^*(\Psi(U(0), \omega)) \quad \mathbb{P}\text{-a.s.}
\end{equation}
In other words, projecting the process $U(\cdot)$ through $\pi_p^*$ is the same as projecting the initial configuration $U(0)$ through $\pi_p$ and running the discrete dynamics.

To justify \eqref{eq: commutative}, we will use only outcomes in the event
\begin{equation}\label{eq: infinite_E_def}
\mathsf{E} = \bigcap_{T \in \mathbb{N}} \left\{\text{there is no infinite chronological path for } [0,T]\right\}.
\end{equation}
By \eqref{eq: infinite_chronological}, one has $\mathbb{P}(\mathsf{E})=1$. First note that two processes from \eqref{eq: commutative} agree at time 0. In other words, for any $x\in \mathcal{V}$, one has 
\[
\left(\Phi(\pi_p(U(0)),\omega)\right)_x(0) = \pi_p(U(0))_x = \left(\pi_p^*(\Psi(U(0),\omega))\right)_x(0).
\]
Next we construct a chronological path backward in time. Let $x \in \mathcal{V}$ and suppose for a contradiction that the two processes disagree at $x$ at time $t$. Set $x_0=x$ and let $t_0 > 0$ be the time of the first clock ring before (or at) time $t$ at $x_0$. (There must be one, since the processes agree at time $0$.) Given vertices $x_0, \ldots, x_k$ such that $x_i$ is a neighbor of $x_{i+1}$ for all $i$ and times $t_0 > t_1 > \cdots > t_k$, where $t_i$ is a time of a clock ring for $x_i$ for all $i$, we consider the neighbors of $x_k$. If the two processes disagree at some neighbor of $x_k$ at time $t_k^-$, then we define $x_{k+1}$ as any such disagreement neighbor and let $t_{k+1} > 0$ be the time of the first clock ring before time $t_k$ at $x_{k+1}$. Note that a.s. $t_{k+1} < t_k$. Because the constructed path is a chronological path for the interval $[0, t]$ and our outcome is in $\mathsf{E}$, the path must be finite. Supposing the path has $k$ many vertices, we see that the two processes agree at the neighbors of $x_k$ at time $t_k^-$.

%Since a spin value is constant between clock rings, we must therefore show that if a vertex $x$ has a clock ring at time $t$, then one has
%\begin{equation}\label{eq: commutative_t}
%\left(\Phi(\pi_p(U(0)),\omega)\right)_x(t) = \left(\pi_p^*(\Psi(U(0),\omega))\right)_x(t).
%\end{equation}
%This means that the update made in the discrete process at $x$ at time $t$ using the initial configuration $\pi_p(U(0))$ is the same as the projection of the update made in the median process at $x$ at time $t$ using the initial configuration $U(0)$. By local finiteness and induction, we may suppose that the discrete process at $x$ and its neighbors at time $t^-$ is equal to the projection of the median process at $x$ and its neighbors at time $t^-$.

There are two cases to consider. First suppose that at least two neighbors of $x_k$ have continuous spin value $\leq p$. In the projected configuration, they have discrete spin $+1$, so when $x_k$'s clock rings, it will assume the spin $+1$. In the median process, $x_k$ will assume a continuous spin $\leq p$, which corresponds also to the spin $+1$. In the other case, at most one neighbor of $x_k$ has continuous spin value $\leq p$, and here, the other two have discrete spin $-1$ in the projected configuration. So when $x_k$'s clock rings, it will assume the spin $-1$. In the median process, $x_k$ will assume a continuous spin $>p$, which corresponds also to the spin $-1$. This means the two processes agree at $x_k$ at time $t_k$, which contradicts our construction of the chronological path. Therefore \eqref{eq: commutative} holds.

To finish the proof of the lemma, we must show that if $B$ is a Borel set in $\{-1,+1\}^{\mathcal{V} \times [0,\infty)}$, then
\begin{equation}\label{eq: to_show_measure_theory}
\mathbb{P}_p(\sigma(\cdot) \in B) = \mathbb{P}((\pi_p(U(t))_{t \geq 0} \in B).
\end{equation}
This equation is restated using the notation in \eqref{eq: commutative} as
\[
\mathbb{P}_p( \Phi(\sigma(0), \omega) \in B) = \mathbb{P}(\pi_p^*(\Psi(U(0),\omega)) \in B).
\]
By \eqref{eq: commutative}, the right side is $\mathbb{P}(\Phi(\pi_p(U(0)),\omega) \in B)$, so we must show that
\[
\mathbb{P}_p( \Phi(\sigma(0), \omega) \in B) = \mathbb{P}(\Phi(\pi_p(U(0)),\omega) \in B).
\]
But the distribution of $(\pi_p(U(0)),\omega)$ under $\mathbb{P}$ is equal to the distribution of $(\sigma(0),\omega)$ under $\mathbb{P}_p$, so this equation holds, and shows \eqref{eq: to_show_measure_theory}. This completes the proof. 
\end{proof}

We next give various properties of the spin variables $U_x(t)$ as $t \to \infty$.
\begin{lem}\label{lem: continuous_limit}
The following hold for the median process.
\begin{enumerate}
\item Almost surely, each $U_x(t)$ has a limit: for any $x \in \mathcal{V}$,
\[
\mathbb{P}\left(U_x(\infty):= \lim_{t \to \infty} U_x(t) \text{ exists}\right) =1.
\]
\item For any $p \in [0,1]$, 
\[
\theta(p) = \mathbb{P}(U_o(t) \leq p \text{ for all large } t)
\]
and so $\mathbb{P}(U_o(\infty) \in [p_c,1-p_c]) = 1$.
\item Almost surely, each $U_x(\infty)$ is the median of its neighboring spins:
\[
U_x(\infty) = \text{median}\left\{ U_y(\infty) : \{x,y\} \in \mathcal{E}\right\}.
\]
Therefore each $x$ has at least one neighbor $y$ such that $U_x(\infty) = U_y(\infty)$.
\end{enumerate}
\end{lem}
\begin{proof}
For any $x$ and any $p$, the discrete spin $\sigma_x(t)$ has a limit as $t \to \infty$ a.s.~under the measure $\mathbb{P}_p$ (see \eqref{eq: discrete_limit}). Then, by Lemma~\ref{lem: projection}, the event that
\[
\lim_{t \to \infty} \mathbf{1}_{\{U_x(t) \leq p\}} \text{ exists for all rational } p \in [0,1]
\]
has probability one. But on this event, $\lim_{t \to \infty} U_x(t)$ has a limit. Indeed, if $\liminf_{t\to\infty} U_x(t)$ and $\limsup_{t\to\infty} U_x(t)$ were different, then we could find a rational $p$ strictly between these them, and $\mathbf{1}_{\{U_x(t) \leq p\}}$ would not have a limit. This would be a contradiction, and this shows item 1. 

For item 2, use Lemma~\ref{lem: projection} to write
\[
\theta(p) = \mathbb{P}_p(\sigma_o(\infty) = +1) = \mathbb{P}(U_o(t) \leq p \text{ for all large }t).
\]
Next, if $\mathbb{P}(U_o(\infty) \in [p_c,1-p_c]) < 1$, then by symmetry we can find $\epsilon>0$ such that with positive probability, $U_o(\infty)\in [0,p_c-\epsilon)$. On this event, for all large $t$, one has $U_o(t) <p_c-\epsilon/2$. Choosing any $p \in (p_c-\epsilon/2,p_c)$, and using Lemma~\ref{lem: projection} for this $p$, we find that with positive probability under $\mathbb{P}_p$, the discrete spin $\sigma_o(t)$ is equal to $+1$ for all large $t$. But this implies that $\sigma_o(\infty) = +1$ with positive probability under $\mathbb{P}_p$, a contradiction since $p<p_c$.

Item 3 is a consequence of the fact that each vertex has infinitely many clock rings a.s. After each clock ring at $x$, the spin at $x$ is equal to the median of its neighboring spins. Therefore the set of times $t$ at which $U_x(t)$ is the median of the values $U_y(t)$ for $\{x,y\} \in \mathcal{E}$ is a.s. unbounded. But because each continuous spin has a limit as $t \to \infty$, we conclude the statement of item 3.
\end{proof}

\subsection{Proof of Theorem~\ref{thm: continuity} given Theorem~\ref{thm: continuous_theorem}}\label{sec: conditional_proof}

%The main result of this section, stated in Theorem~\ref{thm: continuous_theorem}, is that in the median process, the continuous spins flip only finitely often.
%\begin{thm}\label{thm: continuous_theorem}
%Almost surely, each $x$ flips only finitely often: for any $x \in \mathcal{V}$,
%\[
%\mathbb{P}\left(U_x(t^-) \neq U_x(t) \text{ for infinitely many }t\right)=0.
%\]
%\end{thm}
%We first show why it implies Theorem~\ref{thm: continuity}.

\begin{proof}
Assume that the statement of Theorem~\ref{thm: continuous_theorem} is true. Then a.s.\ $o$ flips finitely often, so by \eqref{eq: infinite_chronological}, a.s., one has $U_o(\infty) = U_y(0)$ for some $y \in \mathcal{V}$.

Therefore, the random variable $U_o(\infty)$ has absolutely continuous distribution, since for any $B \subset \mathbb{R}$ of Lebesgue measure zero,
\[ 
\P( U_o(\infty) \in B) \le \sum_{ y \in \mathcal{V} } \P( U_y(0) \in B ) = 0.
\]
% for any set $B$ of Lebesgue measure zero. This implies that $ p \mapsto \P( U_o(\infty) \le p)$ is absolutely continuous. Since any $p \in [0,1]$ is a point of continuity for the law of $U_o(\infty)$, by item 1 of Lemma~\ref {lem: continuous_limit} and Lemma~\ref{lem: projection}, we have
By Theorem~\ref{thm: continuous_theorem} and item 2 of Lemma~\ref{lem: continuous_limit}, we have
 \begin{equation}\label{eq:median_spin_infty}
   \P( U_o(\infty) \le p) =  \mathbb{P}(U_o(t) \leq p \text{ for all large }t) = \theta(p). 
   \end{equation}
So $\theta$ is also absolutely continuous and Theorem~\ref{thm: continuity} holds.
\end{proof}

\begin{rem}\label{rem: new_remark}
Equation~\eqref{eq:median_spin_infty} shows  that for any fixed $p$, the configuration $\pi_p(U(\infty))$ has the same distribution under $\mathbb{P}$ as $\sigma(\infty)$ does under $\mathbb{P}_p$.

%Assuming Theorem~\ref{thm: continuous_theorem}
%
% The reason is that the set of vertices with different values in the configurations $\lim_{t \to \infty}\pi_p(U(t))$ (which has the same distribution as does $\sigma(\cdot)$ by Lemma~\ref{lem: projection}) and $\pi_p(U(\infty))$ are those $x$ with $U_x(\infty)=p$. However, a.s. no $x$ has $U_x(\infty) = p$. Indeed, if a vertex $x$ has $U_x(\infty) = p$, then Theorem~\ref{thm: continuous_theorem} implies that some $y$ must have $U_y(0) = p$. So
%\[
%\mathbb{P}(U_o(\infty) = p) \leq \sum_{y \in \mathcal{V}} \mathbb{P}(U_y(0) = p) = 0.
%\]
\end{rem}

The main work lies in the proof of Theorem~\ref{thm: continuous_theorem}, which will be split over the next few subsections.

\subsection{Agreement clusters}
%Using the limiting continuous spins from Lemma~\ref{lem: continuous_limit}, we define agreement clusters, whose geometric properties will be closely tied to whether continuous spins flip infinitely often.
%\begin{df}\label{def: agreement}
%The agreement graph is the graph whose vertex set is $\mathcal{V}$ and whose edge set consists of those edges $\{x,y\} \in \mathcal{E}$ such that $U_x(\infty) = U_y(\infty)$. For $x \in \mathcal{V}$, the agreement cluster of $x$, written $\mathcal{A}_x$, is the component of the agreement graph containing $x$.
%\end{df}

Recall that the agreement graph is the subgraph of $\mathbb{T}_3$ generated by the agreement edges --- those edges $\{x,y\}$ such that $U_x(\infty) = U_y(\infty)$.
%\begin{rem}\label{rem: connected_clusters}
%On our graph $\mathbb{T}_3$, one can equivalently define the agreement cluster of $x$ as the subgraph of $\mathbb{T}_3$ whose vertex set $V(\mathcal{A}_x)$ consists of those $y \in \mathcal{V}$ such that $U_y(\infty) = U_x(\infty)$, and whose edge set $E(\mathcal{A}_x)$ consists of those edges whose endpoints are in $V(\mathcal{A}_x)$. The reason is that any two vertices $w,z$ with $U_w(\infty) = U_z(\infty)$ must be connected by a path of vertices $u$ with $U_u(\infty) = U_w(\infty)$. In fact, a stronger statement is true: for any finite time $t$, a set of vertices whose spin values are equal must be connected. This is false on other graphs, for instance on the $5$-regular tree $\mathbb{T}_5$. Indeed, if a vertex $x$ at time $t$ has two neighbors with spin value equal to $U_x(t)$, and three other neighbors with strictly larger spin value, it can update its value to one of the strictly larger ones, disconnecting the set of like-spin vertices.
%\end{rem}
The remarkable property of agreement clusters is that they encode exactly when spins flip infinitely often in the median process. Indeed, a spin flips infinitely often if and only if its agreement cluster is infinite. The following result is valid for more general odd-degree Cayley graphs including odd-degree regular trees.
\begin{prop}\label{prop: infinite_equivalence}
The following events are a.s.~equal.
\begin{enumerate}
\item $U_o(t^-) \neq U_o(t)$ for only finitely many $t$.
\item $\mathcal{A}_o$ is finite.
\end{enumerate}
\end{prop}
\begin{proof}
We first define, for $x \in \mathcal{V}$ and a given $M>0$, the event $E_x$ that (a) $U_x(\infty) = U_y(\infty)$ for some $y \in \mathcal{V}$ with $\text{dist}(x,y) \leq M$ and (b) $\mathcal{A}_x$ is infinite. Suppose for a contradiction that with positive probability, item 1 holds, but item 2 does not hold. Then we can choose $M>0$ such that $\mathbb{P}(E_o)>0$. An application of the mass transport principle shows that a.s., each infinite agreement cluster contains either infinitely many $x$ such that $E_x$ occurs or no such $x$. Indeed, we may define a mass transport $m(x,y)$ for $x,y \in \mathcal{V}$ to be 1 if $y$ is the closest vertex of $\mathcal{A}_x$ (using some invariant ordering to break ties, as in Section~\ref{sec: tools}) such that $E_y$ occurs, and to be 0 otherwise. On the event that some agreement cluster is infinite but contains only finitely many $x$ for which $E_x$ occurs, there exists $y$ such that $\sum_{x \in \mathcal{V}}m(x,y) = \infty$. If this event has positive probability, then $\mathbb{E}\sum_{x \in \mathcal{V}} m(x,o) = \infty$, and by the mass transport principle, $\mathbb{E}\sum_{y \in \mathcal{V}}m(o,y) =\infty$. But this integrand is bounded a.s.~by 1, and this gives a contradiction. 

We conclude that under our assumption that with positive probability, item 1 holds, but item 2 does not hold, one also has that with positive probability, $\mathcal{A}_o$ is infinite and contains infinitely many vertices $x$ such that $E_x$ occurs. On this event, choose two vertices $x_1$ and $x_2$ in $\mathcal{A}_o$ such that (a) $\text{dist}(x_1,x_2) > 2M$ and (b) $E_{x_i}$ occurs for $i=1,2$. Then by definition of $E_{x_i}$, we can also find vertices $y_1,y_2$ such that $\text{dist}(y_i,x_i) \leq M$ and $U_{x_i}(\infty) = U_{y_i}(0)$ for $i=1,2$. However $y_1$ cannot equal $y_2$ since $\text{dist}(y_1,y_2) \geq \text{dist}(x_1,x_2) - 2M > 0$, so this implies that $U_{y_1}(0) = U_{y_2}(0)$ for distinct vertices $y_1,y_2$, and this event has zero probability. This is a contradiction, and shows that a.s., item 1 implies item 2.

Now we will show that a.s. item 2 implies item 1. We claim first that a.s., there exists a (random) $T>0$ such that for any $t\geq T$ if $o$ flips at time $t$, then $o$ takes the spin value of a neighbor which is also in $\mathcal{A}_o$. In other words, there is $T>0$ such that for any $t\geq T$, if $U_o(t^-) \neq U_o(t)$, then $U_o(t) = U_y(t)$ for some $y \in \mathcal{A}_o$ which is a neighbor of $o$. The claim is obvious unless $o$ has a neighbor which is not in $\mathcal{A}_o$. In that case, we use the fact that a.s., each $U_y(t)$ has a limit as $t \to \infty$ and that any $y$ which is a neighbor of $o$ but is not in $\mathcal{A}_o$ has $U_y(\infty) \neq U_o(\infty)$. Then for $t$ large enough, one has $U_y(t) \neq U_o(t)$ and the claim holds.

Using the claim, if $\mathcal{A}_o$ is finite, then we can a.s. select $T>0$ such that for all vertices $x$ of $\mathcal{A}_o$, if $x$ flips at time $t \geq T$, then it takes the spin value of a vertex of $\mathcal{A}_o$. This implies that for $t\geq T$, $U_o(t)$ can take only one of the finitely many values in the set $\{U_x(T): x \text{ is a vertex of } \mathcal{A}_o\}$. But since $U_o(t)$ a.s. converges as $t \to \infty$, it must therefore be eventually constant in $t$, meaning that item 1 holds. This shows that item 2 a.s. implies item 1 and completes the proof.
\end{proof}

Proposition~\ref{prop: infinite_equivalence} shows that to prove Theorem~\ref{thm: continuous_theorem} (and therefore Theorem~\ref{thm: continuity}), it suffices to show that a.s., all agreement clusters are finite. As a first step, we show that infinite agreement clusters, if they exist, cannot contain vertices of degree one. 
%We will use this result later to show that if infinite agreement clusters exist, we can make a local modification to ``attach'' a degree-one vertex to such a cluster, and obtain a contradiction. We will then conclude that there are no infinite agreement clusters. 
The following proof seems to require every vertex to have degree 3, though the other results up to this point work on odd-degree trees. 
%{\color{red} this result also appears not to use that agreement cluster is connected}
\begin{lem}\label{lem: no_degree_one}
Almost surely, if $o$ has only one neighbor in $\mathcal{A}_o$, then $\mathcal{A}_o$ is finite.
\end{lem}
\begin{proof}
We use a version of the Nanda-Newman-Stein argument \cite[Theorem~2]{NNS}. For $x \in \mathcal{V}$ and $t \geq 0$, we define the energy 
\begin{equation}\label{eq: natural_energy}
E_x(t) = \sum_{y \in \partial x} \mathbf{1}_{\{U_x(\infty) \neq U_y(\infty)\}};
\end{equation}
 that is, the number of edges of disagreement incident to $x$ at time $t$. Each time a spin flips, the energy $E_x(t)$ may change. Write $m_t(x,y)$ as the total contribution to $E_x(t)-E_x(0)$ caused by flips at $y$. For example, if $y$ flips three times until time $t$ and changes the energy at $x$ by amounts $1, -1, 1$, then $m_t(x,y)$ is $1$. We can then compute, using invariance of our process under automorphisms of $\mathbb{T}_3$,
\begin{align*}
\mathbb{E}\left[E_x(t) - E_x(0)\right] &= \mathbb{E}m_t(x,x) + \sum_{y\neq x} \mathbb{E}m_t(x,y) \\
&= \mathbb{E}m_t(x,x) + \mathbb{E} \sum_{y\neq x} m_t(y,x).
\end{align*}
However, $\sum_{y\neq x} m_t(y,x) = m_t(x,x)$, so we obtain
\begin{equation}\label{eq: mass_identity}
\mathbb{E}\left[ E_x(t) - E_x(0) \right] = 2\mathbb{E}m_t(x,x).
\end{equation}
The left side is bounded by 6, so we find $\left| \mathbb{E}m_t(x,x) \right| \leq 3$ for all $t$.

We claim that each time a vertex $x$ flips, it cannot increase $E_x(t)$. We can argue for this by cases: if $x$ has three or two disagreement edges before a flip, it will take the value $U_y(t)$ of one of its neighbors $y$, and therefore not increase its number of incident disagreement edges. If $x$ has one or zero incident disagreement edges, then $U_x(t)$ must already equal the median of its neighbors values, and therefore will not flip. From this claim we see that $m_t(x,x)$ is monotone in $t$, so by the monotone convergence theorem, $m_\infty(x,x) := \lim_{t \to \infty} m_t(x,x)$ satisfies $|\mathbb{E}m_\infty(x,x)| \leq 3$, and thus $m_\infty(x,x)$ is a.s. finite. However, each time $x$ flips and strictly decreases its own energy, $m_t(x,x)$ decreases by at least 1. We conclude that for any $x \in \mathcal{V}$, a.s.,
\begin{equation}\label{eq: finite_decrease}
x \text{ strictly decreases its number of incident disagreement edges only finitely often.}
\end{equation}

Now assume for a contradiction that with positive probability, $o$ has only one neighbor in $\mathcal{A}_o$ but that $\mathcal{A}_o$ is infinite. Then by Proposition~\ref{prop: infinite_equivalence}, with positive probability, $o$ has only one neighbor in $\mathcal{A}_o$, but $o$ flips at arbitrarily large times. Thus there is a sequence $(t_n)$ of times such that $t_n \to \infty$ and for all $n$, $U_o(t_n^-) \neq U_o(t_n)$. Because $o$ has only one neighbor in $\mathcal{A}_o$, its neighbors $y$ outside of $\mathcal{A}_o$ have $U_y(\infty) \neq U_o(\infty)$, and therefore $U_y(t) \neq U_o(t)$ for all large $t$. This implies that for $n$ large, neither of the values $U_o(t_n^-)$ or $U_o(t_n)$ are equal to any $U_y(t)$ for $y$ a neighbor of $o$ outside of $\mathcal{A}_o$. Thus exactly one of $U_o(t_n^-)$ or $U_o(t_n)$ equals $U_z(t_n)$, where $z$ is the neighbor of $o$ in $\mathcal{A}_o$. Because $o$ can never strictly increase its number of incident disagreement edges, we find that $U_o(t_n) = U_z(t_n)$ and $o$ therefore flips at time $t_n$ from having three incident disagreement edges to two. This is true for all large $n$, and so contradicts \eqref{eq: finite_decrease}. Therefore, a.s., if $o$ has only one neighbor in $\mathcal{A}_o$, then $\mathcal{A}_o$ must be finite.
\end{proof}

The method of proof of Lemma~\ref{lem: no_degree_one} also shows that for all large $t$, the spin at a given vertex must agree with the spin of at least one of its neighbors.
\begin{lem}\label{lem: agree_with_neighbor}
For any $x$, there is a (random, a.s.~finite) time $T_x$ such that for $t\geq T_x$, there is at least one neighbor $y = y(t)$ of $x$ such that $U_y(t) = U_x(t)$. 
\end{lem}
\begin{proof}
Define a sequence of stopping times by $\tau_0=0$ and
\[
\tau_k = \inf\{t \geq \tau_{k-1}+1 : U_x(t) \neq U_y(t) \text{ for all neighbors }y \text{ of } x\} \text{ for } k \geq 1,
\]
where $\inf \emptyset = \infty$. Let $A_t$ be the event that the number of disagreement edges incident to $x$ strictly decreases in the interval $(t,t+1)$ due to a flip of $x$. Writing $\mathcal{F}_{\tau_k}$ as the $\sigma$-algebra of the past before time $\tau_k$, by the strong Markov property, one has for some $c>0$ independent of $k$,
\begin{equation}\label{eq: markov}
\mathbb{P}\left( A_{\tau_k} \mid \mathcal{F}_{\tau_k}\right) \mathbf{1}_{\{\tau_k < \infty\}} \geq c \mathbf{1}_{\{\tau_k < \infty\}}.
\end{equation}
Indeed, for the event $A_{\tau_k}$ to occur, since $x$ disagrees with all neighbors at time $\tau_k$, the clock of $x$ should ring in the interval $(\tau_k,\tau_{k+1})$, while none of the clocks of the neighbors of $x$ ring. This has positive probability independent of $k$. 

We claim then that a.s. on the event $\{\tau_k < \infty \text{ for all } k\}$, $x$ strictly decreases its number of disagreement edges infinitely many times. To see why, we use standard Markov chain arguments: for any $n,N \geq 1$, by \eqref{eq: markov},
\begin{align*}
\mathbb{P}(\cap_{k=n}^{\infty} (A_{\tau_k}^c  \cap \{\tau_{k+1} < \infty\})) &\leq \mathbb{P}(\cap_{k=n}^{n+N} (A_{\tau_k}^c \cap \{\tau_{k+1} < \infty\})) \\
&= \mathbb{E}\left[ \mathbb{P}(A_{\tau_{n+N}}^c \cap \{\tau_{n+N+1} < \infty\} \mid \mathcal{F}_{\tau_{n+N}}) \mathbf{1}_{\{\cap_{k=n}^{n+N-1} (A_{\tau_k}^c \cap \{\tau_{k+1} < \infty\})\}}\right] \\
&\leq (1-c) \mathbb{P}\left( \cap_{k=n}^{n+N-1} (A_{\tau_k}^c \cap \{\tau_{k+1} < \infty\})\right).
\end{align*}
Continuing in this way, we find that the left side is bounded above by $(1-c)^{N}$ for each $N\geq 1$, so it is zero. This establishes that for any $n \geq 1$, a.s., on the event $\{\tau_k < \infty \text{ for all }k\}$, at least one event of the form $A_{\tau_k}$ occurs for some $k \geq n$. Intersecting over $n$ gives the claim.

By \eqref{eq: finite_decrease} and the above claim, we obtain $\mathbb{P}(\tau_k < \infty \text{ for all } k) = 0$, and this completes the proof. 
\end{proof}

 The last result we will need for our main proofs states that for any $x$, $U_x(t)$ cannot oscillate around its limit $U_x(\infty)$. For its statement, denote $\mathrm{sgn}(u) =-1, 0, 1$ if $u<0,  = 0, > 0$ respectively. 
\begin{lem}\label{l:NNS_finite_osc}
For any $x \in \mathcal{V}$, the limit $\lim_{ t \to \infty} \mathrm{sgn}( U_x(t)  - U_x(\infty) ) $  exists almost surely. 
\end{lem}

\begin{proof}
The lemma will follow if we show that functions $\mathbf{1}_{\{ U_x(t)  \ge  U_x(\infty) \} }$ and $ \mathbf{1}_{ \{U_x(t)  \le  U_x(\infty)\}  }$ have a.s.\ limits as $t \to \infty$.  It suffices to show the existence of an a.s.\ limit for the first function, as the proof for the second one is identical. 

We again employ a version of the Nanda-Newman-Stein argument.  
%If an edge $e = \{x, y\}$  is not an agreement edge, then $U_t(x)  \ne U_t(y)$ eventually.  Denote
%\[ \gamma_e = \inf\{s \ge 0: U_t(x)  \ne U_t(y) \text{ for all } t \ge s \} < \infty, \]
%if $e$ is not an agreement edge and $\gamma_e = 0$ otherwise. For a vertex $x \in \mathcal{V}$,  let
%\[ T_x = \inf\{ s \ge 0: \exists y \sim x \text{ such that } U_y(t) = U_x(t) \text{ for all } t \ge s\}.\]
%By Remark~\ref{rem: degree_one}, a.s., $T_x < \infty$. Now set
%\[ \tau_x =  1+ \max( T_x, \max_{y \in \mathcal{V}:  y \sim x} \gamma_{\{x, y\}})< \infty.\]
For an edge $\{x, y\}$,  define its energy at time $t$ as $E_{\{x, y\}}(t) = \mathbf{1}_{A_t}$ where the event $A_t$ is given by
\[A_t = \big \{  |\mathcal{A}_x| = \infty,  \   \ U_x(\infty) = U_y(\infty) \ \text{ and }  \mathbf{1}_{ \{ U_x(t)  \ge  U_x(\infty)  \} } \ne \mathbf{1}_{ \{ U_y(t)  \ge U_y(\infty) \} } \big \}. \]
Note that the definition of the event $A_t$ is symmetric in $x$ and $y$ as $U_x(\infty) = U_y(\infty)$ implies $\mathcal{A}_x =  \mathcal{A}_y$. 
The energy of a vertex $x$ at time $t$ is now defined as
\[ E_x(t) = \sum_{y \in \mathcal{V}:  y \sim x} E_{\{x, y\}}(t).\]
$E_{\{x, y\}}(t)$ and $E_x(t)$ may change due to a clock ring at $x$ or any of its neighbors. Let us introduce the notations $\Delta_t E_{\{x, y\}}$ and $\Delta_t E_x$ to denote these changes.
\[ \Delta_t E_{\{x, y\}}  = E_{\{x, y\}}(t) -  E_{\{x, y\}}(t^-)  \ \  \text{ and } \Delta_t E_x = E_x(t) - E_x(t^-) =  \sum_{y \in \mathcal{V}:  y \sim x} \Delta_t E_{\{x, y\}}. \ \  \]
 Let $m_t(x, y)$ be the total contribution to $E_x(t) - E_x(0)$ induced by the rings at $y$ during time $(0, t]$, i.e., 
\[ m_t(x, y) = \sum_{s \le t: \text{ $y$ has a ring at time } s }  \Delta_s E_x.\]
Similar to how we argued in the proof of Lemma~\ref{lem: no_degree_one} for equation~\eqref{eq: mass_identity}, using the invariance of the process, we obtain
\begin{equation}\label{eq:NNS_inv}
 \E[ E_x(t) - E_x(0) ] = 2 \E[ m_t(x, x)].
 \end{equation}

Next we make two important claims about how the energy of the vertex $x$ evolves with time.

\noindent {\bf Claim 1.} A flip at $x$ cannot strictly increase the energy of $x$. That is, if $U_x(t) \ne U_x(t^-)$, then $\Delta_t E_x \le 0$.

\noindent {\bf Claim 2.}  Let $B$ be the event  on which  the function $\mathbf{1}_{\{ U_x(t)  \ge  U_x(\infty) \} }$  changes its value infinitely often as $t \to \infty$.  Then, a.s.\ $\Delta_t E_x  <0 $ infinitely often on $B$.

Assuming the above claims, it is easy to conclude the proof of Lemma~\ref{l:NNS_finite_osc}. Note that Claim 1 implies that $t \mapsto - m_t(x,x) \ge 0$ is non-decreasing.  It follows from Claim 2 that $ - m_t(x,x) \uparrow \infty$ a.s.~on $B$ as $t \to \infty$. Therefore, if  $\prob(B)>0$, then 
by the monotone convergence theorem, we obtain $ \E[ m_t(x, x) ]  \downarrow  - \infty$. Since $E_x(t) \in \{0,1,2,3\}$, it follows from \eqref{eq:NNS_inv} that $ \E[ m_t(x, x)] \le \tfrac32$ for all $t$, which would be a contradiction. Hence, $B$ must have zero probability, as desired.

We will prove Claims 1 and 2 simultaneously. We assume that $|\mathcal{A}_x | = \infty$. This is because if $\mathcal{A}_x $ is finite,  then $E_x (t) \equiv 0$ and Claim 1 trivially follows. Also, on the event $B$,  $x$ must flip infinitely often and hence by Proposition~\ref{prop: infinite_equivalence}, a.s.~$\mathcal{A}_x$ must be infinite. If $\mathcal{A}_x$ is infinite, then by Lemma~\ref{lem: no_degree_one}, $\mathrm{deg}_{ \mathcal{A}_x}(x)$, the degree of $x$ in the agreement graph, can not be one. So, $\mathrm{deg}_{ \mathcal{A}_x}(x)$ is either $2$ or $3$. We will handle these two cases separately.

\noindent \textbf{Case I:} $\mathrm{deg}_{ \mathcal{A}_x}(x) =2.$

Let $y_1, y_2, z$ be three neighbors of $x$ such that $y_1, y_2 \in \mathcal{A}_x$ and $z \not \in \mathcal{A}_x$.   Denote $u = U_x(\infty) = U_{y_1}(\infty) =  U_{y_2}(\infty)$. Suppose the vertex $x$ flips at time $t$, and that the energy of one of the two agreement edges incident to $x$, say  $\{x, y_1\}$,  increases from $0$ to $1$ at $t$ due to that flip. For definiteness,  assume that $U_{y_1}(t^-) = U_{y_1}(t)  \ge u$ (the other case is similar). 
Since $E_{\{x, y_1\}}(t^-) = 0$ and $E_{\{x, y_1\}}(t) =1$, we must have $U_x(t^-) \ge u$ and $U_x(t) < u$. That means at time $t^-$,  both $y_2$ and $z$ should satisfy $U_{y_2}(t^-) = U_{y_2}(t) < u$, and  $U_z(t^-) = U_z(t) < u.$
Consequently, 
\[ E_{\{x, y_2\}}(t^-)  =1, \ \   E_{\{x, y_2\}}(t) = 0, \ \ \text{ and }  \ \ \Delta_t E_{\{x, y_2\}}  = -1.\]
The vertex $z$, not being in the agreement cluster of $x$,  does not contribute to the energy of $x$. Therefore,  if $x$ has a flip at time $t$ and $\Delta_t E_{\{x, y_1\}} =1$, then 
$\Delta_t E_x =  \Delta_t E_{\{x, y_1\}} + \Delta_t E_{\{x, y_2\}} = 1  -1 = 0$. So,  $\Delta_t E_x \le 0$ always in the case. 

Next we are going to argue that on the event $B$, $\mathrm{deg}_{ \mathcal{A}_x}(x) \ne 2$, so Claim 2 vacuously holds in this case.  Since $z \not \in \mathcal{A}_x$,   $U_z(\infty) \ne u$. Assume, without loss of generality, that  $U_z(\infty) >  u$.
It follows  that $U_z(t) >   U_x(t)$  for all large $t$. Thus by Lemma~\ref{lem: agree_with_neighbor}, $U_x(t) \in \{ U_{y_1}(t), U_{y_2}(t)\}$  for all large $t$. On the event $B$, we can find arbitrarily large times $t$ at which $x$ flips from a value  at or above $u$ to a value below $u$, i.e., $U_x(t^-) \ge u$ and $U_x(t)<  u$. But at time $t^-$, one of the vertices from $\{y_1, y_2\}$ and the vertex $z$ have values at least $u$. So, the updated value $U_x(t)$  of $x$, can not go below $u$, which is a contradiction.

\noindent \textbf{Case II:} $\mathrm{deg}_{ \mathcal{A}_x}(x) =3.$

Let $y_1, y_2, y_3 \in \mathcal{A}_x $ be the three neighbors of $x$ and denote  $u = U_x(\infty) = U_{y_i}(\infty) $ for $i=1, 2, 3$.  Suppose the vertex $x$ flips at time $t$.  If either both $U_x(t^-)$ and $U_x(t)$ are  $\ge u$ or both are $< u$, the energy of each edge $\{x, y_i\}$ remains unchanged and hence $\Delta_t E_x  = 0$.  Now suppose that  $U_x(t^-) \ge u$ and $U_x(t) < u$. Then there must be at least two neighbors, say $y_2, y_3$ whose values at $t$ are $< u$.  Then $\Delta_t E_{\{x, y_2\}} = \Delta_t E_{\{x, y_3\}} = -1$ and
\[ \Delta_t E_x   = \sum_{i=1}^3 \Delta_t E_{\{x, y_i\}}  \le -1.\]
The same conclusion can be reached if $U_x(t^-) < u$ and $U_x(t) \ge u$. So, we observe that if $x$ flips at time $t$, then $ \Delta_t E_x  \le -1$ and  $\Delta_t E_x  = 0$ depending on whether $1_{ \{ U_x(t)  \ge  u  \} } $ changes its value at $t$ or not. It establishes  both Claims 1 and 2 in this case and thereby finishing the proof of the lemma.
\end{proof}

\subsection{Proof of Theorem~\ref{thm: finite_agreement_clusters}}\label{sec: agreement_finite}

\subsubsection{Geometry of infinite agreement cluster} 

If  $\mathcal{A}_o$ is infinite, then  by Proposition~\ref{prop: infinite_equivalence}, almost surely, every vertex in $\mathcal{A}_o$ flips infinitely many times.  Therefore, as a consequence of Lemma~\ref{l:NNS_finite_osc}, any vertex $v \in \mathcal{A}_o$ satisfies either $U_v(t) < U_v(\infty)$ for all large times $t$ OR $U_v(t) > U_v(\infty)$ for all large time $t$. (The reason is that $\text{sgn}(U_v(t)-U_v(\infty))$ is otherwise eventually 0, and this means that $v$ stops flipping.) Depending on these two cases, we call a vertex $v$ (in some infinite agreement cluster)   as type $<$ or type $>$ respectively. When $\mathcal{A}_o$ is infinite, it can thus be partitioned  into subgraphs $\mathcal{A}_o^{<}$ and $\mathcal{A}_o^{>}$, where $\mathcal{A}_o^{<}$ (resp. $\mathcal{A}_o^{>}$) is generated by all type $<$ (resp. type $>$) vertices of $\mathcal{A}_o$. If $\mathcal{A}_o$ is infinite with positive probability, then either $\mathcal{A}_o^{<}$ or $\mathcal{A}_o^{>}$ is non-empty with positive probability. Without loss of generality, let us assume that $\mathcal{A}_o^{<} \ne \emptyset$ with positive probability. 

We now claim that if $\mathcal{A}_o^{<}$  is non-empty, then every vertex in $\mathcal{A}_o^{<}$  has at least two neighbors in $\mathcal{A}_o^{<}$. Note that if $x \in \mathcal{A}_o^{<}$ and $y$ is a neighbor of $x$  such that $ y \not \in \mathcal{A}_o^{<}$, then $U_x(t) \ne U_y(t)$ for all large $t$. By Lemma~\ref{lem: agree_with_neighbor},  $U_x(t)$ must share the spin of at least one of the neighbors of $x$  at all large times $t$.  So, at least one neighbor of $x$ must be in $\mathcal{A}_o^{<}$. If $x$ has exactly one neighbor $y \in \mathcal{A}_o^{<}$, then $U_x(t) = U_y(t)$ for all large times $t$ and this forbids $x$ from flipping infinitely many times, contradicting the beginning assumption that $\mathcal{A}_o = \mathcal{A}_x$ is infinite.

The union of the subgraphs $\mathcal{A}_x^{<}$ for $x \in \mathcal{V}$  induces an invariant bond percolation on $\mathbb{T}_3$. From the discussion above, each (non-empty) open component $\mathcal{A}_x^{<}$ of this percolation contains at least one doubly infinite path. We deduce that (see \cite[Theorem~1.2]{Haggstrom3})  almost surely, either 
\begin{itemize}
\item[(i)] $\mathcal{A}_o^{<}$  has two ends, i.e., $\mathcal{A}_o^{<}$  is a single doubly infinite path, OR,  
\item[(ii)] $\mathcal{A}_o^{<}$  has infinitely many ends. 
\end{itemize}
In the next subsection we rule out the first possibility above.

\subsection{Ruling out the case when $\mathcal{A}_o^{<}$ has two ends}\label{sec: two_ends}
In this section we will assume that with positive probability,  $\mathcal{A}_o^{<}$ is a single doubly infinite path and arrive at a contradiction. Without loss of generality, we can assume that $o \in \mathcal{A}_o^{<}$. 

\noindent \textbf{Claim 1.} If $x \in \mathcal{A}_o^{<}$ and $y$ is a neighbor of $x$ with  $y \not \in \mathcal{A}_o^{<}$, then  $U_x(t) < U_y(t)$ for all large $t$.  

If possible, suppose that $U_x(\infty) > U_y(\infty)$. Let $x_1$ and $x_2$ be the two neighbors of $x$ in $\mathcal{A}_o^{<}$. Therefore, by Lemma~\ref{lem: agree_with_neighbor}, $U_x(t) \in \{ U_{x_1} (t), U_{x_2} (t) \}$ for all large $t$.  Since $U_x(\infty) > U_y(\infty)$,  we have $U_x(t) > U_y(t)$ and $U_{x_i}(t) > U_y(t), i = 1, 2$ for all large $t$.   Moreover, if $x$ flips at a large time $t$, then $U_x(t) = \min (U_{x_1} (t^-), U_{x_2} (t^-)) = \mathrm{median}\{ U_{x_1} (t^-), U_{x_2} (t^-), U_y(t^-)\}$. So, every flip of $x$ after a sufficiently large time results in a decrease in the value of the spin at $x$. This is in contraction with the facts that $U_x(t) \to U_x(\infty)$ and $U_x(t) < U_x(\infty)$ for all large $t$. So,  we must have  $U_x(\infty) \le U_y(\infty)$.  If $U_x(\infty) =  U_y(\infty)$, then  $y \in \mathcal{A}_o^{>}$. Otherwise, $U_x(\infty) < U_y(\infty)$. In both cases, the claim is obvious. 

Deleting a vertex $x$  splits $\mathbb{T}_3$  into three disjoint infinite binary subtrees $(\mathbb{T}_{y \to x}: y \in \partial x)$.  Let us denote the three neighbors of $o$ by $x_{-1}, x_{+1}, $ and $x_0$. Then the following event has a positive probability: 
\begin{equation}  \label{eq:descrip_event1}
\begin{gathered}
\text{ $\mathcal{A}_o^{<}$ is a single doubly infinite path  passing through $o$  in the subtree  $\mathbb{T}_{o \to  x_0}$}  \\
 \text{ and contains no vertex from $\mathbb{T}_{x_0 \to  o}$. }
\end{gathered}
\end{equation}
Let $\eta = (\eta_x)_{x \in \mathbb{T}_3} $ denote all the randomness in the median process including both the initial (continuous) spins and the clocks.  We write $\eta = (\eta_1, \eta_2)$ where $\eta_1 =  (\eta_x)_{x \in \mathbb{T}_{o \to  x_0}}$ and $\eta_2 =  (\eta_x)_{x \in \mathbb{T}_{x_0 \to  o}}$. 
Let $\eta' = (\eta_1, \eta_2')$ be obtained from $\eta$ by independently resampling the initial spins and the clocks for the vertices in $\mathbb{T}_{x_0 \to  o}$, while keeping them unchanged for the vertices in $\mathbb{T}_{o \to  x_0}$. Clearly, $\eta'$ shares the same law as $\eta$. Let $\Omega \ni \eta$ be the original sample space for $\eta$ and $\tilde \Omega \in \tilde \eta := (\eta_1, \eta_2, \eta_2')$ be the enlarged sample space containing  the extra randomness $\eta_2'$. 
Let us denote $Z(\eta) = U_\infty(o)(\eta) \in [p_c, 1-p_c]$. Let $\sigma^Z$ be the $\{ -1, +1\}^{\mathcal{V}}$-valued process on $\Omega$ obtained by thresholding the median process $U$  at $Z$,  i.e.,  the discrete spin $\sigma^Z_v(t)$ of a vertex $v$ at time $t$ is given by 
\[ \sigma^Z_v(t)(\eta) = +1 \  \text{ if  \ \ } U_t(v)(\eta) \le Z(\eta)  \quad \text{ and } \quad  \sigma^Z_v(t)(\eta) = -1 \text{  \ \ otherwise}. \]
Note that $\sigma^Z$ follows the majority dynamics with initial discrete spins $\sigma^Z_v(0), v \in \mathbb{T}_3$. 
Almost surely in $\sigma^Z$, the discrete spin of every vertex $v$ only flips finitely often and hence,  $\sigma^Z_v(\infty)   =  \lim_{t \to \infty} \sigma^Z_v(t)$ exists a.s.\ for each vertex $v$. This follows directly from Lemma~\ref{l:NNS_finite_osc} if $U_v(\infty) = Z$, while for  $U_v(\infty) \ne Z$,  $U_v(t)$ can not oscillate around $Z$ infinitely many times because $U_v(t) \to U_v(\infty)$. So, from \eqref{eq:descrip_event1} and Claim 1, with positive probability, 
\begin{equation}\label{eq:descrip_event_2}
\begin{gathered}
 \text{the limiting configuration $(\sigma^Z_v(\infty))_{v \in \mathbb{T}_3}$ contains  a single doubly infinite path of $+1$ spins  }  \\
 \text{  passing through $o$  in the subtree  $\mathbb{T}_{o \to  x_0}$  and $\sigma^Z_{x_0}(\infty)  = - 1$.}
\end{gathered}
\end{equation}
Let us denote the above event by $\mathcal{B} \subseteq \Omega$. 
Now we run a median process  corresponding to randomness $\eta'$ and threshold it at $Z = Z(\eta)$ to obtain a `perturbed' majority dynamics $\hat \sigma^Z$  on the enlarged probability space $\tilde \Omega$:
\[ \hat \sigma^Z_v(t)(\tilde \eta) = +1 \  \text{ if  \ \ } U_t(v)(\eta') \le Z(\eta)  \quad \text{ and } \quad  \sigma^Z_v(t)(\tilde \eta) = -1 \text{  \ \ otherwise}. \]
We are going to show that

\noindent \textbf{Claim 2.}  There exists $\mathcal{B}_0 \subseteq \mathcal{B} \subseteq \Omega $ such that $\mathcal{B}_0$ has positive $\eta$-probability and for a.s.\ every $\eta \in \mathcal{B}_0$,  with positive $\eta_2'$-probability, 
\begin{equation} \label{eq:descrip_event_3}
\begin{gathered}
 \sigma^Z_v(t)(\eta) = \hat \sigma^Z_v(t)(\tilde \eta) \text{  for all $t$ and  for $v \in \mathbb{T}_{o \to  x_0}$  } \\
\text{ and $ \hat \sigma^Z_{x_0}(t)(\tilde \eta) = +1 $ for all large $t$.}
\end{gathered}
\end{equation}
We will later arrive at a contradiction by showing that structures of this kind are forbidden due to invariance. \\

\noindent \textbf{Proof of Claim 2.}  Suppose that the clock at $o$ rings at times $0< r_1 < r_2< \ldots $  in $\eta$ (and hence in $\tilde \eta$ as well). Note that the processes  $\sigma^Z(\eta)$ and  $\hat \sigma^Z(\tilde \eta)$  will remain identical on $\mathbb{T}_{o \to  x_0}$ if  $\sigma^Z_{x_o}(r_i) (\eta) =  \hat \sigma^Z_{x_o}(r_i) (\tilde \eta)$ for each $i$; that is,  if the spin of the boundary vertex $x_o$ is the same in both processes whenever the clock of $o$ rings.

First,  by considering a subevent  of $\mathcal{B}$ with positive probability if necessary, we can assume that  there exists a real number $q \in (0, 1)$  such that 
\[ q \le Z  \text{ on }   \mathcal{B} \quad \text{ and } \quad \theta(q) > 0.\]
 Indeed, if  $Z = p_c$ a.s.\ on $\mathcal{B}$, then we take $q = p_c$. Since $U_o(t) \le Z=p_c$ for all large $t$ on $\mathcal{B}$, it follows that $\theta(p_c) = \lim_{ t \to \infty} \prob( U_o(t) \le p_c) \ge \prob( \mathcal{B})   >0$. Else, suppose that $Z > p_c$ occurs with positive probability on $\mathcal{B}$. Then we can find $q > p_c$ such that  $q < Z$  occurs with a  positive probability on some sub-event of $\mathcal{B}$. Since $q > p_c$, $\theta(q) > 0$ by definition. 

Let  $\mathcal{B}_0$ be the following sub-event of  $\mathcal{B}$, which occurs with positive probability. 
\begin{enumerate}
\item[(I)] after a fixed large time $T$, the $\sigma^Z$ spins of three vertices $o, x_{-1}, x_{+1}$  never flip and remain at $+1$ forever. 
\item[(ii)]  the clock of $o$ rings exactly $k$ times before $T$ and the $i$th ring occurs in the time interval   $[s_i, t_i], 1 \le i \le k$ where $0 < s_1 < t_1 < \cdots < s_k < t_k <  s_{k+1} := T$ are given deterministic numbers,  
\item[(iii)]  the clock of $x_0$ does not ring (in $\eta_2)$ during each of $k$ time intervals $[s_i, t_i],$ and  $\sigma^Z_{x_0}(t)(\eta) = a_i$ for all $t \in [s_i, t_i]$ for some given deterministic $a_1, a_2, \ldots, a_k \in \{-1, 1\}$.  
\end{enumerate}

In the remainder of the proof, we will argue that  if $\eta \in \mathcal{B}_0$, then with positive $\eta_2'$-probability, irrespective of (a priori potentially unfavorable) $\hat \sigma^Z$ dynamics on $\mathbb{T}_{o \to x_0}$, 
 we can ensure that 
 \begin{equation}\label{eq:descrip_event_4}
  \hat \sigma^Z_{x_0}(t)(\tilde \eta) = a_i \text{  for all } t \in [s_i, t_i]  \ \ \text{ and }  \ \  \hat \sigma^Z_{x_0}(t)(\tilde \eta)  = +1 \text{  for all } t \ge T.
  \end{equation}
 From the discussion in the beginning of the proof of Claim 2, the above condition will then guarantee \eqref{eq:descrip_event_3}. 
 
Consider the binary tree $\mathbb{T}_{x_0 \to o}$ rooted at $x_0$. For any vertex $v \ne x_0$ in $\mathbb{T}_{x_0 \to o}$, let  $v^{-}$ be the parent of $v$, i.e., the unique vertex in $\partial v$ closest to $x_0$. Let $L_i$ be the set of vertices of $\mathbb{T}_{x_0 \to o}$ at distance $i$ from $x_0$ and set $L = \cup_{i=0}^{k} L_i$. 

By choosing $\eta_2'$ suitably, we can set the initial discrete spin of each vertex $v_i \in L_i$ to satisfy $\hat \sigma^Z_{v_i}(0)(\tilde \eta)   =a_i$   for $i=0, 1, \ldots, k+1$, where $a_{k+1}=+1$. Indeed, this can be achieved by taking the initial continuous spin at a vertex $v_i \in L_i$ to satisfy either $U_v(0)(\eta_2') < p_c$ or  $U_v(0)(\eta_2') > 1-p_c$, depending on whether $ a_i = 1$ or $-1$ respectively. Of course, this can be done with positive $\eta_2'$-probability. 
By Lemma~\ref{lem:never_flip}, for each $v_{k+1} \in L_{k+1}$, we can further choose (in $\eta_2')$ the initial values and clocks of the vertices in the binary trees $\mathbb{T}_{v_{k+1} \to v_{k+1}^-}$  in such a way that
with positive $\eta_2'$-probability, 
\[ U_{v_{k+1}}(t)(\eta_2') < q \text{  for all } t.\]
 (Note that Lemma~\ref{lem:never_flip} gives $\leq q$, but since $q$ is fixed, there is zero probability that any spin equals $q$ at any finite time.) Therefore,  the discrete spins of the vertices in $L_{k+1}$ remain frozen at $+1$  at all time in $\hat \sigma^Z$.   

We now arrange the clocks of the vertices in $L$ in a suitable way to ensure the correct boundary condition \eqref{eq:descrip_event_4} at the vertex $x_0$. No clock rings for any vertex in $L$ within each of the time intervals $[s_i, t_i]$. Also, there is no clock ring in $L$ before time $s_1$. Within each of the $k$ intervals $(t_i, s_{i+1})$ for $i = 1, 2, \ldots, k$, first the vertices in $L_k$ ring once (in some order), then the vertices in $L_{k-1}$ ring once, and so on, all the way up to the root $x_0$. Since the updates in $\hat \sigma^Z$ follow the majority rule and each vertex in $L_i$ has two neighbors from $L_{i+1}$, every such cycle of clock rings during the interval $(t_i, s_{i+1})$  has the effect of pushing all the spins  in $L$ by one level up (and popping out the current spin of $x_0$). 
This will ensure that we have the right boundary condition at $x_0$ at each interval $[s_i, t_i]$ for $1 \le i \le k$. Upon completion of the final cycle in $(t_k, T)$, all vertices in $L$ in $\hat \sigma^Z$  permanently become $+1$, satisfying the condition $ \hat \sigma^Z_{x_0}(t)(\tilde \eta)  = +1 \text{  for all } t \ge T$. This special arrangement of clocks of vertices in $L$ can be made with positive $\eta'$-probability.  This concludes the proof of Claim 2. 

\begin{lem}\label{lem:single_path_invariance}
The following event $\mathcal{C}$ concerning  the median process on $\mathbb{T}_3$ has zero probability.  There exists a random variable $Z$ (defined on an enlarged probability space)  with the properties 
\begin{itemize}
\item[(i)] there exists a unique doubly infinite path $S$, passing through $o$  in the subtree  $\mathbb{T}_{o \to  x_0}$, such that every vertex $v$ on that path satisfies $U_v(t) \le Z$ for all large $t$.
\item[(ii)]  each vertex $u$ in $ \partial S \cap \mathbb{T}_{o \to  x_0}$  satisfies $U_u(t) > Z$ for large $t$, where $\partial S$ denotes the (external) vertex boundary of the path $S$.
\item[(iii)] $U_{x_0}(t) \le Z$ for all large $t$.
\end{itemize}
\end{lem}
To apply Lemma~\ref{lem:single_path_invariance}, we consider the median process on $\mathbb{T}_3$ with randomness $(\eta_1, \eta_2')$ and let $Z(\tilde \eta) = U_o(\infty)(\eta)$, all defined on the enlarged probability space $\tilde \Omega$.
If $\mathcal{A}_o^{<}$ is a single doubly infinite path with positive probability, then by Claim 2, $\prob(\mathcal{C}) > 0$, contradicting Lemma~\ref{lem:single_path_invariance}. Therefore, $\mathcal{A}_o^{<}$ can not be a single doubly infinite path a.s., as promised.  
\begin{rem}
In the above lemma, the process $(1_{ \{U_v(\cdot) \le Z\} })_{v \in \mathbb{T}_3}$ may not be invariant on $\mathbb{T}_3$. Had it been the case, the proof of the lemma would have followed easily using standard results for invariant percolation on trees. 
\end{rem}

\begin{proof}[Proof of Lemma~\ref{lem:single_path_invariance}]   
Observe that $U_{u}(\infty) = U_{v}(\infty)$ for any $u, v \in S$ on $\mathcal{C}$. If not, suppose $U_{v_1}(\infty) < U_{v_2}(\infty)$ for a pair of adjacent vertices $v_1$ and $v_2$ on $S$ and assume that $v_1 \ne o$. Let $v_1' \in \partial S$ be the neighbor of $v_1$.  Then $U_{x}(t) > U_{v_1}(t)$ for large $t$ for both $x = v_1'$ and $x = v_2$. This contradicts the median update at $v_1$ at any clock ring of $v_1$ after a large time. So, we deduce  that all the vertices on the same side of $o$ (perhaps not including $o$) share the same limiting spin value. Now fix a neighbor $x_1 \in \partial o \cap  S$ of $o$ and suppose $U_{x_1}(\infty) \ne U_o(\infty)$.   Let $x_2$ and $x_1'$ be the two other neighbors of $x_1$, on and off the path $S$.  From Lemma~\ref{lem: agree_with_neighbor}, we must have $U_{x_1}(t) = U_{x_2}(t)$ for large $t$ and consequently, $x_1$ cannot flip at a large time. This is a contradiction since 
the agreement cluster of $x_1$ is infinite on the given event and hence by Proposition~\ref{prop: infinite_equivalence}, $x_1$ has to flip infinitely often.

Now we proceed to prove the lemma by considering several cases. Let $V :=  U_o(\infty)$ be the common limiting spin of the vertices on the path $S$.

\noindent Case I :  $V  = Z$.  For any vertex $v$, let $G_v$ be the connected component of $v$ in the subgraph of $\mathbb{T}_3$ consisting of all vertices $u$ satisfying $U_u(t) \le U_v(\infty)$ for large $t$. Since vertices on $S$ share the same limiting spin, we have $G_{v_1} = G_{v_2}$ for any $v_1, v_2 \in S$. Furthermore, $v \in G_v$ for all $v \in S$ since $V = Z$. Note that for $v \in S$, $G_v \cap \mathbb{T}_{o\to x_0}$ is the exactly the doubly infinite path $S$, as any vertex $u$ in $\mathbb{T}_{o\to x_0} \cap \partial S$   satisfies  $U_u(t) > Z$ for large $t$ by (ii).
 Also,  since $U_{x_0}(t) \le Z=V$ for all large $t$, we deduce that $x_0 \in G_v$ for any $v \in S$. Consequently, each $v\in S$ has degree $2$ in $G_v$, except for $o$ which has degree $3$ in $G_o$. 
 
 We define a mass transport by sending  unit mass from $x$ to $y$ if $y$ is the unique vertex (if any) with degree $3$ in  $G_x$, and zero mass otherwise. Then the root $o$ receives an infinite mass in total (as all the vertices on the path $S$ will send  unit mass to $o$) on the event $\mathcal C  \cap \{V=Z\}$. So, if Case I happens with positive probability,  then the expected total mass received by $o$ is infinite, whereas the expected total mass sent out by $o$ is always bounded above by $1$. This contradicts the mass transport principle!

\noindent Case II : $V < Z$.  Then the agreement cluster $\mathcal{A}_{o}$ in $\mathbb{T}_{o \to x_0}$ is precisely the doubly infinite path $S$, since the limiting spin of every vertex in $\partial S \cap \mathbb{T}_{o \to x_0}$ is at least $Z$.  We are going to treat two cases separately depending on whether $U_{x_0}(\infty) = V$ or $U_{x_0}(\infty) \ne V$. 

\noindent Case II(a): $ U_{x_0}(\infty) = V$.   Then $x_0 \in \mathcal{A}_o$ and by  Lemma~\ref{lem: no_degree_one}, $x_0$ belongs to an infinite connected component of $\mathcal{A}_o$  in $\mathbb{T}_{x_0 \to o}$. Thus, $o$ is a triple point of $\mathcal{A}_o$ and hence, $\mathcal{A}_o$ has at least three ends. (Recall that  for a connected subgraph $G$ of $\mathbb{T}_3$, we say a vertex $x$ of $G$ is a triple point of $G$, if the removal of $x$ and its incident edges from $G$ splits it  into three infinite components.)

Moreover, $\mathcal{A}_o$ has at least  two isolated ends (coming from the doubly infinite path $S$). Since the agreement graph  is an invariant percolation on $\mathbb{T}_3$ and $\mathcal{A}_o$ is one of its connected components,  $\mathcal{A}_o$ must have  infinitely many ends (see  \cite[Theorem~1.2]{Haggstrom3}), in which case it can not have an isolated end a.s.\ (see \cite[Proposition~1.4]{Haggstrom3}).   So, Case II(a) cannot occur with positive probability. 

\noindent Case II(b): $U_{x_0}(\infty) \ne V$.   Then $\mathcal{A}_o$ consists of the doubly infinite path $S$.

Assume first that $U_{x_0}(\infty) < Z$.  Note that every vertex $v \in \partial S$, except for $x_0$, has limiting spin $U_v(\infty) \ge Z$. This makes $o$ the unique vertex on $S$ whose neighbor $x_0 \in \partial S$ has the limiting spin satisfying $ U_{x_0}(\infty)  =  \inf_{ v \in \partial S} U_v(\infty)$.  Construct a mass transport where we send  unit mass from $x$ to $y$ if $y \in \mathcal{A}_x$ and $y$ is the closest vertex to $x$ (if any) among the vertices $v \in \mathcal{A}_x$ with a neighbor $v' \in \partial \mathcal{A}_x$ satisfying $U_{v'}(\infty)  = \inf_{ u' \in \partial \mathcal{A}_x} U_{u'}(\infty)$. If the current case happens with positive probability, then the expected total mass received by $o$  is infinite, while the expected total mass out is bounded by $1$, leading to a contradiction.

It remains to consider the case $U_{x_0}(\infty) = Z$.  For any vertex $x$, let $I_x  = \inf_{ w \in \partial \mathcal{A}_x} U_w(\infty) $. Clearly, for any $v \in \mathcal{A}_o$, $I_v = \inf_{ w \in \partial S} U_w(\infty) =  U_{x_0}(\infty) = Z$. Note that by (iii), $U_{x_0}(t) \le Z = U_{x_0}(\infty) $ for all large $t$ %{\color{green} do we need the sgn claim here?}.  
On the other hand,  for any other vertex $v \in \partial S$,  $U_{v}(t) > Z  = U_{x_0}(\infty)  $ for large $t$ by (ii).

Now we define a mass transport by sending unit mass from $x$ to $y$ if $y \in \mathcal{A}_x$ and $y$ is the closest vertex to $x$ (if any) among the vertices  $v \in \mathcal{A}_x$ with the property that it has a neighbor $v' \in \partial \mathcal{A}_x$  such that $U_{v'}(t) < I_x$  for all large $t$. Assuming this scenario happens with positive probability,  we again arrive at the contradiction by the mass transport principle since the expected total mass received by $o$ will be infinite, whereas the expected total mass sent by $o$ will be bounded by $1$. 

\end{proof}

\subsection{Ruling out the case when $\mathcal{A}_o^{<}$ has infinitely many ends}

In this section, we assume that 
\begin{equation}\label{eq: infinitely_many_ends_assumption}
\mathbb{P}(\mathcal{A}_o^< \text{ has infinitely many ends})>0, 
\end{equation}
and  eventually obtain a contradiction. This will show that both possibilities listed above Section~\ref{sec: two_ends} are impossible, and therefore that a.s., $\mathcal{A}_o$ is finite. Therefore we can go back through Proposition~\ref{prop: infinite_equivalence} to finally prove Theorem~\ref{thm: continuous_theorem} and consequently Theorem~\ref{thm: continuity}.

To prove that $\mathcal{A}_o^<$ cannot have infinitely many ends, we use Lemma~\ref{cor: percolation}. For this, we need to define invariant bond percolation processes $\eta, (\eta_n)$ on $\mathbb{T}_3$ satisfying its hypotheses. The definition of $\eta$ will require a simple extension of the definition of $\mathcal{A}_o^<$: for any $x \in \mathcal{V}$, we write $\mathcal{A}_x^<$ for the subgraph induced by the vertices $y$ in $\mathcal{A}_x$ with $U_y(t) < U_y(\infty)$ for all large $t$. Note that if $y$ is a vertex of $\mathcal{A}_x$, then $\mathcal{A}_x^< = \mathcal{A}_y^<$. We then define $\eta \in \{0,1\}^{\mathcal{E}}$ by
\[
\eta(e) = \begin{cases}
1 &\quad\text{if } e = \{x,y\} \text{ and } e \text{ is an edge of } \mathcal{A}_x^< \\
0 & \quad\text{otherwise}.
\end{cases}
\]
Therefore the edges with $\eta$-value 1 are those whose endpoints are in the same agreement cluster and whose spin values are (for all large times) less than their limiting values. This $\eta$ is an invariant bond percolation process.

To define $\eta_n$, we will, for a given $n$, choose $t_n$ so large that 
\[
\mathbb{P}(U_o(t_n) < U_o(\infty) \mid o \text{ is a vertex of } \mathcal{A}_o^<) > 1-\frac{1}{n^2}.
\]
Next, choose $\delta_n>0$ so that
\begin{equation}\label{eq: delta_n_fix}
\mathbb{P}(U_o(t_n) < U_o(\infty)-\delta_n \mid o \text{ is a vertex of } \mathcal{A}_o^<) > 1-\frac{2}{n^2}.
\end{equation}
Then we define the invariant bond percolation process $\eta_n$ as
\[
\eta_n(e) = \begin{cases}
1 &\quad\text{if } e = \{x,y\} \text{ is an edge of } \mathcal{A}_x^< \\
& \qquad \text{and } \max\{U_x(t_m),U_y(t_m)\} < U_x(\infty)-\delta_m \text{ for all } m \geq n \\
0 & \quad \text{otherwise}.
\end{cases}
\]

Note that $\eta_n \leq \eta_{n+1} \leq \eta$ for all $n$ and using \eqref{eq: delta_n_fix}, one has for any $e = \{x,y\}$,
\begin{align*}
&\mathbb{P}(\eta_n(e)=0 \mid \eta(e) = 1) \\
\leq~& 2\sum_{m=n}^\infty \mathbb{P}(U_x(t_m) \geq U_x(\infty)-\delta_m \mid \eta(e) = 1) \\
\leq~& 2 \sum_{m=n}^\infty \mathbb{P}(U_x(t_m) \geq U_x(\infty)-\delta_m \mid x \text{ is a vertex of }  \mathcal{A}_x^<) \cdot \frac{\mathbb{P}(x \text{ is a vertex of } \mathcal{A}_x^<)}{\mathbb{P}(\eta(e)=1)} \\
\leq~& \frac{2\mathbb{P}(x \text{ is a vertex of } \mathcal{A}_x^<)}{\mathbb{P}(\eta(e) = 1)} \sum_{m=n}^\infty \frac{2}{m^2} \\
\leq~&  \frac{2\mathbb{P}(x \text{ is a vertex of } \mathcal{A}_x^<)}{\mathbb{P}(\eta(e)=1)} \cdot \frac{2}{n-1}.
\end{align*}
Therefore $\lim_{n\to\infty} \mathbb{P}(\eta_n(e)=1)= \mathbb{P}(\eta(e)=1)$ and we can invoke Lemma~\ref{cor: percolation} to conclude that if $\mathbb{P}(A_\eta(o),D_\eta(o))>0$, then
\begin{equation}\label{eq: limit_ends}
\lim_{n\to\infty} \mathbb{P}(D_{\eta_n}(o) \mid A_\eta(o),D_\eta(o)) =1,
\end{equation} 
where $A_{\eta}(o)$ is the event that $C_{\eta}(o)$, the $\eta$-open cluster of $o$, has at least three ends, and $D_\eta(o)$ is the event that $o$ is in a self-avoiding doubly-infinite $\eta$-open path (similarly for $D_{\eta_n}(o)$), as defined before Lemma~\ref{cor: percolation}. Note that because we assumed \eqref{eq: infinitely_many_ends_assumption}, the probability of $A_\eta(o)$ is positive. Furthermore, as we have noted, each vertex $x \in \mathcal{A}_x^<$ has degree at least two in $\mathcal{A}_x^<$, so $\mathbb{P}(D_\eta(o)\mid A_\eta(o))=1$, giving that $\mathbb{P}(A_\eta(o),D_\eta(o)) =\mathbb{P}(A_\eta(o)) >0$, and so \eqref{eq: limit_ends} holds.

For any outcome in $D_{\eta_n}(o)$, the root $o$ is in a doubly infinite path $P$ such that all vertices $x \in P$ satisfy $U_x(t_n) < U_x(\infty)-\delta_n$ and for all $x,y \in P$, one has $U_x(\infty) = U_y(\infty)$. By the definition of the median process, such vertices are ``stable'' in the sense that for all times larger than $t_n$, whenever they attempt to update their spins, they have at least two neighbors with spin values less than $U_x(\infty) - \delta_n$. Therefore any $x \in P$ must have $U_x(\infty) \leq U_x(\infty) - \delta_n$, which is a contradiction unless $\mathbb{P}(D_{\eta_n}(o))=0$ for all $n$. Applying this in \eqref{eq: limit_ends} gives a contradiction, and shows that \eqref{eq: infinitely_many_ends_assumption} must be false, completing the proof.

\section{Correlation decay: proof of Theorem~\ref{thm: correlation_decay}}\label{sec: correlation_decay}
To prove the correlation decay statement, Theorem~\ref{thm: correlation_decay}, we first study the trace of a vertex $x$ in the median process. This is the set of vertices which ever assume the spin value that is initially at $x$. The main result is that a.s., the trace of a vertex is finite. We use this along with a coupling trick to show the perturbation result, Theorem~\ref{thm: resample_spin}: for Lebesgue a.e.~$p$, in the discrete spin model with initial density $p$ of $+1$, changing one initial spin from $+1$ to $-1$ (or vice-versa) only affects finitely many spins for all time. In other words, a single spin has only a finite range of influence. From this statement, we argue the correlation decay.

\begin{df}\label{def: trace}
For $x \in \mathcal{V}$, the trace of $x$ is the set of vertices
\[
\text{Tr}(x) = \{y \in \mathcal{V} : U_y(t) = U_x(0) \text{ for some } t \geq 0\}.
\]
\end{df}
The main result about the trace is the following. It was previously stated as Theorem~\ref{thm: finite_trace}.
\begin{prop}\label{prop: finite_trace}
Almost surely, one has $\# \text{Tr}(o) < \infty$.
\end{prop}
\begin{proof}
The idea of the proof is that if $\# \text{Tr}(o) = \infty$ with positive probability, then with positive probability, at a large time $t$ there will also be a vertex $x$ with $U_x(t) = U_o(0)$, and hence $x \in \text{Tr}(o)$,  such that $x$ has stopped flipping (that is, $U_x(s) = U_x(\infty)$ for all $s \ge t$). Next, on the event $\# \text{Tr}(o) = \infty$, we can find an infinite self-avoiding path starting at the root and consisting of vertices $o, x_1, x_2, x_3, \ldots$ such that each $x_n \in \text{Tr}(o)$.
%Letting $(x_n)$ be a sequence of vertices with $dist(o,x_n) \to \infty$ such that $U_{x_n}(t_n) = U_o(0)$ for some times $t_n$,
Let $(t_n)$ be an increasing sequence of times such that $U_{x_n}(t_n) = U_o(0)$. We then have a.s.\ $t_n \to \infty$ (because of finiteness of chronological paths). Furthermore one can check that for any $t \geq 0$ and vertices $w,z$, if $U_w(t) = U_z(t)$, then the direct path in $\mathbb{T}_3$ between $w$ and $z$ must also have all vertices with spin equal to $U_w(t)$ at time $t$. (This is false on other graphs, like $\mathbb{T}_5$.) Since for large $n$, $U_x(t_n) = U_{x_n}(t_n)$, the spins of the vertices on the path from $x$ to $x_n$ at time $t_n$ are all equal to $U_o(0)$. This implies that the agreement cluster of $x$ at time infinity is infinite, and this event has zero probability. 

To carry out this argument, define, for $t\geq 0$, the mass transport
\[
m_t(x,y) = \begin{cases}
1 & \quad \text{if } \# \text{Tr}(x) = \infty \text{ and } U_y(t) = U_x(0) \\
0 & \quad \text{otherwise}.
\end{cases}
\]
Note that for any fixed $t$, on the event that $\{\#\text{Tr}(o) = \infty\}$, there exists at least one vertex $y$ such that $U_y(t) = U_o(0)$. 
Then by the mass transport principle, for any $t$,
\begin{align*}
\mathbb{P}(\#\text{Tr}(o) = \infty) &\leq \mathbb{E}\left[ \mathbf{1}_{\{\#\text{Tr}(o) = \infty\}} \sum_{y \in \mathcal{V}} \mathbf{1}_{\{U_y(t) = U_o(0)\}} \right] \\
&= \mathbb{E}\sum_{y \in \mathcal{V}} m_t(o,y) \\
&= \mathbb{E}\sum_{x \in \mathcal{V}} m_t(x,o) \\
&= \mathbb{E}\#\{x \in \mathcal{V} : \#\text{Tr}(x) = \infty \text{ and } U_o(t) = U_x(0)\} \\
&= \mathbb{P}(\exists ~x \in \mathcal{V} : \#\text{Tr}(x) = \infty \text{ and } U_o(t) = U_x(0)) \\
&\leq \mathbb{P}(\exists ~x \in \mathcal{V} : \#\text{Tr}(x) = \infty \text{ and } U_o(\infty) = U_x(0)) \\
&+ \mathbb{P}(U_o(\infty) \neq U_o(t)).
\end{align*}
Letting $t \to \infty$ and using that vertices flip only finitely often a.s. (a combination of Proposition~\ref{prop: infinite_equivalence} and Theorem~\eqref{thm: finite_agreement_clusters}), we obtain
\[
\mathbb{P}(\#\text{Tr}(o) = \infty) \leq \mathbb{P}(\exists ~x \in \mathcal{V} : \#\text{Tr}(x) = \infty \text{ and } U_o(\infty) = U_x(0)).
\]
As discussed at the beginning of the proof, the right side is bounded above by $\mathbb{P}(\exists x \in \mathcal{V} : \#\mathcal{A}_x = \infty)$, which is zero by Theorem~\eqref{thm: finite_agreement_clusters}.
\end{proof}

The next step in the proof of Theorem~\ref{thm: correlation_decay} is to relate the finiteness of the trace in the continuous spin model to the influence of the initial value of a single spin $\sigma_x(0)$ on the evolution in the discrete spin model. Given an initial spin configuration and realization of Poisson clocks (dynamics) $(\sigma(0),\omega)$, we define the discrete evolution $\sigma^+(\cdot)$ given by the usual majority vote dynamics specified in \eqref{eq: energy_rules}, but with initial configuration $\sigma^+(0)$ which equals $\sigma(0)$ except at the root $o$, where $\sigma_o^+(0) = +1$. Similarly we define the evolution $\sigma^-(\cdot)$ by setting $\sigma_o^-(0) = -1$. For two evolutions $\sigma(\cdot),\tau(\cdot) \in \{-1,+1\}^{\mathcal{V}\times [0,\infty)}$, recall the definition of the symmetric difference
\[
\sigma(\cdot) \Delta \tau(\cdot) = \{y \in \mathcal{V} : \sigma_y(t) \neq \tau_y(t) \text{ for some } t \geq 0\}
\]
as the set of vertices that have different spins in the evolutions $\sigma$ and $\tau$ at some time $t$. The following lemma was previously stated as Theorem~\ref{thm: resample_spin}.
\begin{lem}\label{lem: resample_spin}
For Lebesgue-a.e.\ $p \in [0,1]$, one has
\begin{equation}\label{eq: resample_spin}
\mathbb{P}_p\left(\# \sigma^+(\cdot) \Delta \sigma^-(\cdot) < \infty \right) = 1.
\end{equation}
\end{lem}
\begin{proof}
%By symmetry, we may take $p \leq 1/2$. Furthermore, the case $p=0$ is trivial, so we may assume that
%\[
%p \in (0,1/2].
%\]
The proof will be a direct consequence of finiteness of trace in the continuous model, after applying a certain projection. For a given evolution of continuous spins $U(\cdot)$, we define two projected evolutions $\tau^+(\cdot)$ and $\tau^-(\cdot)$ in $\{-1,+1\}^{\mathcal{V}\times[0,\infty)}$ by
\[
\tau_x^+(t) = \begin{cases}
+1 & \quad\text{if } U_x(t) \leq U_o(0) \\
-1 & \quad\text{if } U_x(t) > U_o(0)
\end{cases}
\]
and
\[
\tau_x^-(t) = \begin{cases}
+1 & \quad\text{if } U_x(t) < U_o(0) \\
-1 & \quad\text{if } U_x(t) \geq U_o(0).
\end{cases}
\] 
These are projections relative to the initial value of the root. Note that by construction
\[
\tau^+(\cdot) \Delta \tau^-(\cdot) = \text{Tr}(o),
\]
so by Proposition~\ref{prop: finite_trace},
\begin{equation}\label{eq: one_half_projection}
\mathbb{P}\left( \# \tau^+(\cdot) \Delta \tau^-(\cdot) < \infty \right) = 1.
\end{equation}
%By Fubini's theorem, this implies that
%\begin{equation}\label{eq: one_half_projection_two}
%\mathbb{P}\left( \# \tau^+(\cdot) \Delta \tau^-(\cdot) < \infty \mid U_o(0) = p\right) = 1 \text{ for Lebesgue-a.e. }p.
%\end{equation}
%(The measure $\mathbb{P}(\cdot \mid U_o(0)=p)$ is the regular conditional probability measure.)

We next must compare the pairs of processes $(\tau^+(\cdot),\tau^-(\cdot))$ and $(\sigma^+(\cdot),\sigma^-(\cdot))$. The idea is that because, conditional on $U_o(0)$, the families $(\tau_x^+(0):x \neq o)$ and $(\tau_x^-(0):x \neq o)$ are each i.i.d.~with $\mathbb{P}(\tau_x^{\pm}(0) = 1) = U_o(0) = 1-\mathbb{P}(\tau_x^{\pm}(0)=-1)$, the pair $(\tau^+(\cdot),\tau^-(\cdot))$ has the same distribution as $(\sigma^+(\cdot),\sigma^-(\cdot))$ does under $\widetilde{\mathbb{P}}$. Here $\widetilde{\mathbb{P}}$ is a measure on pairs $(\sigma(0), \omega)$, where the two entries are independent, $\omega$ is a realization of rate-one Poisson clocks, $\sigma(0) = (\sigma_x(0))_{x \in \mathcal{V}}$ is an i.i.d.~collection with $\mathbb{P}(\sigma_x(0)=1) = Y = 1-\mathbb{P}(\sigma_x(0)=-1)$, and $Y$ is an independent random variable with uniform distribution on $[0,1]$. So the lemma should follow from \eqref{eq: one_half_projection} and Fubini's theorem.
%
%
%so the lemma should follow from \eqref{eq: one_half_projection} when $p=1/2$. To handle general $p$, we will need to consider a conditioned measure $\mathbb{P}(\cdot \mid U_o(0) \leq r)$, where $r$ is selected so that the conditional distribution of the families $\tau^+(0)$ and $\tau^-(0)$ are each i.i.d. with common distribution Bernoulli$(p)$. To select $r$, we note that since $U_x(0)$ and $U_o(0)$ are i.i.d. uniform $[0,1]$, one has
%\begin{equation}\label{eq: integral}
%\mathbb{P}(U_x(0) \leq U_o(0) \mid U_o(0) \leq r) = \frac{1}{r} \cdot \int_0^r \int_0^s ~\text{d}u~\text{d}s = \frac{r}{2},
%\end{equation}
%so we put $r = 2p$.

Most of the proof is almost the same as that of Lemma~\ref{lem: projection}, but with more complicated notation, so we omit some details. For discrete spin configurations $\sigma(0), \tau(0) \in \{-1,+1\}^{\mathcal{V}}$ and $\omega$ a realization of Poisson clocks, let $\hat \Phi$ be the map that sends the element $((\sigma(0),\tau(0)), \omega)$ to the pair of discrete evolutions $(\sigma(\cdot), \tau(\cdot))$. As before, for an initial configuration of continuous spins $U(0)$ and $\omega$ a realization of Poisson clocks, let $\Psi$ be the map that sends $(U(0), \omega)$ to the continuous evolution $U(\cdot)$. We also need projection operators, but this time their definitions are more involved. Define the operator
\[
\pi : [0,1]^{\mathcal{V}} \to \{-1,+1\}^{\mathcal{V}} \times \{-1,+1\}^{\mathcal{V}}
\]
by $\pi(\hat U) = (\pi_+(\hat U), \pi_-(\hat U))$, where
\[
\left(\pi_+(\hat U)\right)_x = \begin{cases}
+1 & \quad\text{if } \hat U_x \leq \hat U_o \\
-1 & \quad\text{if } \hat U_x > \hat U_o
\end{cases}
\]
and
\[
\left(\pi_-(\hat U)\right)_x = \begin{cases}
+1 & \quad\text{if } \hat U_x < \hat U_o \\
-1 & \quad\text{if } \hat U_x \geq \hat U_o.
\end{cases}
\]
In words, $\pi$ sends a continuous spin configuration to a pair of two discrete configurations which are projected in two different ways according to the spin at the root $o$. Note that
\begin{equation}\label{eq: equal_initial_conditioned}
\begin{split}
\text{the distribution of } ((\sigma^+(0),\sigma^-(0)),\omega) &~~~ \text{ under } \widetilde{\mathbb{P}} \\
=~\text{the distribution of } (\pi(U(0)),\omega) &~~~ \text{ under } \mathbb{P} 
\end{split}
\end{equation}

We must also extend $\pi$ to evolutions, defining
\[
\pi^* : [0,1]^{\mathcal{V} \times [0,\infty)} \to \{-1,+1\}^{\mathcal{V} \times [0,\infty)} \times \{-1,+1\}^{\mathcal{V} \times [0,\infty)}
\]
by $\pi^*(\hat U(\cdot)) = \left( \pi_+^*(\hat U(\cdot)), \pi_-^*(\hat U(\cdot))\right)$,
where
\[
\left( \pi_+^*(\hat U(\cdot))\right)_x(s) = \begin{cases}
+1 & \quad\text{if } \hat U_x(s) \leq \hat U_o(0) \\
-1 & \quad\text{if } \hat U_x(s) > \hat U_o(0)
\end{cases}
\]
and
\[
\left( \pi_-^*(\hat U(\cdot))\right)_x(s) = \begin{cases}
+1 & \quad\text{if } \hat U_x(s) < \hat U_o(0) \\
-1 & \quad\text{if } \hat U_x(s) \geq \hat U_o(0).
\end{cases}
\]
In words, $\pi^*$ sends a continuous evolution to a pair of discrete evolutions which are projected in two different ways according to the continuous spin at the root $o$ at time $0$. Now, just as in \eqref{eq: commutative}, one can check the relation
\begin{equation}\label{eq: complicated_commutative}
\hat \Phi(\pi(U(0)), \omega) = \pi^*(\Psi(U(0),\omega)) \quad \mathbb{P}\text{-a.s.}
\end{equation}
%By Fubini's theorem, this implies that for Lebesgue-a.e. $p$,
%%Because this is an a.s. statement relative to $\mathbb{P}$, and the event $\{U_o(0) \leq r\}$ has positive probability (as $r \in (0,1]$), we also have
%\begin{equation}\label{eq: second_complicated_commutative}
%\hat \Phi(\pi(U(0)), \omega) = \pi^*(\Psi(U(0),\omega)) \quad \mathbb{P}(\cdot \mid U_o(0) = p)\text{-a.s.}
%\end{equation}
%(The measure on the right is the regular conditional probability measure.)

Next, we must show that for any Borel set $B$ of $\{-1,+1\}^{\mathcal{V}\times [0,\infty)} \times \{-1,+1\}^{\mathcal{V}\times [0,\infty)}$,
\begin{equation}\label{eq: equal_dist_conditioning}
\widetilde{\mathbb{P}}\left( (\sigma^+(\cdot),\sigma^-(\cdot)) \in B\right) = \mathbb{P}\left( (\tau^+(\cdot), \tau^-(\cdot)) \in B\right).
\end{equation}
%Indeed, from this and \eqref{eq: one_half_projection}, the lemma would follow. 
Showing this is similar to showing \eqref{eq: to_show_measure_theory}. We can rewrite both sides in our notation as
\[
\widetilde{\mathbb{P}} \left( \hat \Phi((\sigma^+(0),\sigma^-(0)),\omega) \in B\right) = \mathbb{P}\left( \pi^*(\Psi(U(0),\omega)) \in B \right).
\]
Using \eqref{eq: complicated_commutative}, we can again rewrite it as
\[
\widetilde{\mathbb{P}} \left( \hat \Phi((\sigma^+(0),\sigma^-(0)),\omega) \in B\right) = \mathbb{P}\left(\hat \Phi(\pi(U(0)),\omega) \in B \right).
\]
These are equal because of \eqref{eq: equal_initial_conditioned}, and this shows \eqref{eq: equal_dist_conditioning}.

By \eqref{eq: one_half_projection} and \eqref{eq: equal_dist_conditioning}, one has
\[
\widetilde{\mathbb{P}}(\#\sigma^+(\cdot) \Delta \sigma^-(\cdot) < \infty) = 1,
\]
and so by Fubini's theorem,
\[
\widetilde{\mathbb{P}}(\#\sigma^+(\cdot) \Delta \sigma^-(\cdot) < \infty \mid Y = p) = 1 \text{ for Lebesgue-a.e. } p \in [0,1].
\]
(The measure on the left is the regular conditional probability measure.) However, given $Y=p$, the initial spins $(\sigma_x(0))_{x \neq o}$ are i.i.d.~with distribution $\mathbb{P}(\sigma_x(0) = +1) = p = 1-\mathbb{P}(\sigma_x(0) = -1)$. Because the event in question depends only on the initial spins at vertices not equal to $o$ and all the Poisson clocks, this is the same as
\[
\mathbb{P}_p(\#\sigma^+(\cdot) \Delta \sigma^-(\cdot) < \infty) = 1 \text{ for Lebesgue-a.e.} p \in [0,1].
\]
This completes the proof.
\end{proof}

Lemma~\ref{lem: resample_spin} implies that for Lebesgue-a.e.\ $p$, if we compare two discrete evolutions, one with initial condition $\sigma(0)$ and one with an initial condition equal to $\sigma(0)$ except the initial spin at the root $o$ is resampled, then they will $\mathbb{P}_p$-a.s.~only differ at finitely many vertices of $\mathcal{V}$ for all time. From this we conclude that resampling a finite number of initial spins has a ``finite effect'' on the evolution. For the correlation decay statement of Theorem~\ref{thm: correlation_decay}, we will also need to resample Poisson clocks, so our next lemma deals with this additional resampling.

For an initial configuration $\sigma(0)$ in $\{-1,+1\}^{\mathcal{V}}$ sampled from the Bernoulli product measure $\mathbb{P}_p$ with $p \in [0,1]$, let $X$ be an independent variable with $\mathbb{P}(X = 1) = p = 1-\mathbb{P}(X=-1)$, and define $\sigma'(0) \in \{-1,+1\}^{\mathcal{V}}$ as 
\[
\sigma_x'(0) = \begin{cases}
\sigma_x(0) & \quad \text{if } x \neq o \\
X & \quad \text{if } x=o.
\end{cases}
\]
We similarly define realizations of Poisson clocks, taking $\omega$ as usual as a realization of independent Poisson clocks at all vertices of $\mathcal{V}$, and $\omega'$ equal to $\omega$ except with the clock at $o$ resampled.

\begin{lem}\label{lem: resample_all}
If $\Phi$ is the map that sends an initial configuration $(\sigma(0),\omega)$ to its discrete evolution $(\sigma(t))_{t \geq 0}$, then for Lebesgue-a.e.\ $p \in [0,1]$,
\begin{equation}\label{eq: resample_all}
\mathbb{P}_p(\# \Phi(\sigma(0),\omega) \Delta \Phi(\sigma'(0),\omega') < \infty) = 1.
\end{equation}
\end{lem}
\begin{proof}
%We suppose for a contradiction that
%\begin{equation}\label{eq: for_contradiction_resample}
%\mathbb{P}_p(\#\Phi(\sigma(0),\omega)\Delta \Phi(\sigma'(0),\omega') = \infty)  = c > 0.
%\end{equation}
Fix any $p$ satisfying~\eqref{eq: resample_spin}. We first resample the clock at the root for only a finite amount of time. For a fixed $T>0$, let $\omega'_T$ be the realization of Poisson clocks which equals $\omega'$ until time $T$, and then equals $\omega$ afterward. Note that $\omega'_T$ has the same distribution as $\omega$ or $\omega'$. Furthermore, a.s., for our given pairs $(\sigma(0),\omega)$, $(\sigma'(0),\omega')$, there exists a (random) $T$ large enough so that 
\[
\#\Phi(\sigma'(0),\omega') \Delta \Phi (\sigma'(0),\omega'_T) = 0.
\]
The reason is that from \eqref{eq: discrete_limit}, a.s. in the evolution $\Phi(\sigma'(0),\omega')$, the root $o$ flips only finitely often and must in fact be ``stable'' for all large times. In other words, for all large times, the spin at $o$ agrees with those of at least two of its neighbors, and neither the root nor any of its neighbors flip. If $T$ is taken to be beyond this time, there will be no difference in the evolutions using clocks $\omega'$ or $\omega'_T$. Therefore given $\epsilon>0$, we may fix a deterministic $T>0$ such that
\begin{equation}\label{eq: can_resample_finite_clock}
\mathbb{P}_p\left( \#\Phi(\sigma'(0),\omega') \Delta \Phi(\sigma'(0),\omega'_T) = 0 \right) > 1-\epsilon/2.
\end{equation}

Next we show that if the clock at the root is only resampled for a finite time, the evolution using this clock can be bounded between evolutions using the original clock, but with initial spin configurations that are identically $+1$ or $-1$ on a ball of large radius centered at the root. To do this, we define for $R>0$ configurations $\sigma_{+,R}(0)$ (respectively $\sigma_{-,R}(0)$) in $\{-1,+1\}^{\mathcal{V}}$ to be equal to $\sigma(0)$ for vertices at distance $>R$ from the root and equal to $+1$ (respectively $-1$) for all other vertices. By applying Lemma~\ref{lem: resample_spin} finitely  many times, one has for any $R>0$,
\begin{equation}\label{eq: resample_many_spins}
\mathbb{P}_p\left( \# \Phi(\sigma_{+,R}(0),\omega) \Delta \Phi(\sigma_{-,R}(0),\omega) < \infty\right) = 1.
\end{equation}

Note that by attractiveness (see the end of Section~\ref{sec: tools}), one has for any $R>0$,
\begin{equation}\label{eq: trivial_dominance}
\mathbb{P}_p\left( \Phi(\sigma_{+,R}(0),\omega) \geq \Phi(\sigma(0),\omega) \geq \Phi(\sigma_{-,R}(0),\omega)\right) = 1,
\end{equation}
where $\geq$ means pointwise inequality for all times. We next prove that a similar statement holds if the middle term is replaced by $\Phi(\sigma'(0),\omega'_T)$, so long as $R$ is large: for our fixed $T$, the following inequality holds a.s.\ for all (random) $R$ large enough:
\begin{equation}\label{eq: propagation_dominance}
\Phi(\sigma_{+,R}(0),\omega) \geq \Phi(\sigma'(0),\omega'_T) \geq \Phi(\sigma_{-,R}(0),\omega),
\end{equation}
To show \eqref{eq: propagation_dominance}, note that by attractiveness and the fact that the clock configurations $\omega'_T$ and $\omega$ are equal for times $>T$, one need only show that for all $x \in \mathcal{V}$,
\begin{equation}\label{eq: finite_time_propagation_dominance}
\left(\Phi(\sigma_{+,R}(0),\omega)\right)_x(t) \geq \left(\Phi(\sigma'(0),\omega'_T)\right)_x(t) \geq \left(\Phi(\sigma_{-,R}(0),\omega)\right)_x(t) \quad \text{for all } t \leq T.
\end{equation}
This inequality follows from the finite speed of propagation in the model using clocks $\omega$. Supposing, for example, that the left inequality were to fail, then for some $t \leq T$, one would have $\left(\Phi(\sigma_{+,R}(0),\omega)\right)_o(t) = -1$; otherwise, the spin at the root $o$ in $\Phi(\sigma_{+,R}(0),\omega)$ would always dominate the corresponding spin $\Phi(\sigma'(0),\omega'_T)$, and since the clocks $\omega'_T$ only differ from those in $\omega$ at the root, attractiveness would imply \eqref{eq: propagation_dominance}. Therefore there must be a chronological path for $[0,T]$ starting at a vertex $x$ with $\text{dist}(x,o)=R$ and ending at o.
%along which there are clock rings occurring in succession between times $0$ and $T$. Note we may choose this path to be vertex self-avoiding, and so each such path has $R+1$ vertices, and there are $3\cdot 2^{R-1}$ many of them. For any given path $\Gamma$ of this type, the time it takes for successive clock rings along $\Gamma$ is at least $\sum_{i=1}^{R+1} \tau_i$, where the $\tau_i$'s are i.i.d. exponential mean one random variables. 
A similar statement holds in the case of the right inequality of \eqref{eq: finite_time_propagation_dominance}, so by Lemma~\ref{lem: chronological},
\begin{align*}
\mathbb{P}_p\left( \eqref{eq: finite_time_propagation_dominance} \text{ fails}\right) \leq \frac{5e^{4T}}{4} \left( \frac{4}{5}\right)^{R+1}.
\end{align*}
Since this bound is summable in $R$, the Borel-Cantelli lemma shows that \eqref{eq: finite_time_propagation_dominance} holds a.s.\ for all large $R$. (In fact Borel-Cantelli is not needed here since the events are monotone in $R$ by attractiveness.) Because \eqref{eq: finite_time_propagation_dominance} implies \eqref{eq: propagation_dominance}, one can, for our fixed $T$, choose $R>0$ so large that
\begin{equation}\label{eq: last_inequality_for_now}
\mathbb{P}_p\left( \Phi(\sigma_{+,R}(0),\omega) \geq \Phi(\sigma'(0),\omega'_T) \geq \Phi(\sigma_{-,R}(0),\omega)\right) > 1-\epsilon/2.
\end{equation}

From the bounds above, we see that with probability at least $1-\epsilon$, the intersection of the events in \eqref{eq: can_resample_finite_clock}, \eqref{eq: resample_many_spins}, \eqref{eq: trivial_dominance}, and \eqref{eq: last_inequality_for_now} occurs. But on this intersection, one must have $\#\Phi(\sigma(0),\omega) \Delta \Phi(\sigma'(0),\omega') < \infty$. Indeed, this symmetric difference is contained the following union:
\begin{align*}
\left( \Phi(\sigma(0),\omega) \Delta \Phi(\sigma_{+,R}(0),\omega) \right) &\cup \left(\Phi(\sigma_{+,R}(0),\omega) \Delta \Phi(\sigma'(0),\omega'_T) \right) \\
&\cup \left( \Phi(\sigma'(0),\omega_T') \Delta \Phi(\sigma'(0), \omega')\right).
\end{align*}
The first term is finite on the intersection of events in \eqref{eq: resample_many_spins} and \eqref{eq: trivial_dominance}. The second is finite on the intersection of events in \eqref{eq: resample_many_spins} and \eqref{eq: last_inequality_for_now}, and the third is finite on the event in \eqref{eq: can_resample_finite_clock}. Therefore the probability in the statement of the lemma is at least $1-\epsilon$, and since $\epsilon$ is arbitrary, this completes the proof.
\end{proof}

Given the resampling result, Lemma~\ref{lem: resample_all}, we can now prove Theorem~\ref{thm: correlation_decay}.
\begin{proof}[Proof of Theorem~\ref{thm: correlation_decay}]
Fix any $p$ from \eqref{eq: resample_all}. To show that the strong mixing coefficient $\alpha_{r,R}(p) \to 0$ as $R\to\infty$, we will show that the collections $(\sigma_x(\infty) : \text{dist}(x,o) \geq R)$ and $(\sigma_x(\infty) : \text{dist}(x,o) \leq r)$ are with high probability equal to two independent collections. So let $M>0$ and let $(\sigma(0),\omega)$ and $(\sigma'(0),\omega')$ be two independent realizations of initial spins and Poisson clock configurations, each sampled from $\mathbb{P}_p$. Define the new configurations $(\sigma^{(1)}(0),\omega^{(1)})$ (which will represent resampling outside the ball of radius $M$ centered at $o$) and $(\sigma^{(2)}(0),\omega^{(2)})$ (which will represent resampling inside the ball of radius $M$ centered at $o$) as follows. We set $(\sigma^{(1)}(0),\omega^{(1)})$ equal to $(\sigma(0),\omega)$ at vertices $x$ with $\text{dist}(x,o) \leq M$ and $(\sigma'(0),\omega')$ elsewhere, while we set $(\sigma^{(2)}(0),\omega^{(2)})$ equal to $(\sigma(0),\omega)$ at vertices $x$ with $\text{dist}(x,o) > M$ and $(\sigma'(0),\omega')$ elsewhere. Note that
\begin{equation}\label{eq: initial_independence}
(\sigma^{(1)}(0),\omega^{(1)}) \text{ and } (\sigma^{(2)}(0),\omega^{(2)}) \text{ are independent}.
\end{equation}

To use \eqref{eq: initial_independence}, let $r,R$ be fixed with $0 \leq r \leq R < \infty$, let $A$ be any event in $\Sigma_{\leq r}$, and let $B$ be any event in $\Sigma_{\geq R}$. We can then choose Borel sets $\hat A$ and $\hat B$ depending only on values of spins at vertices $x$ with $\text{dist}(x,o) \leq r$ and $\text{dist}(x,o) \geq R$ respectively so that
\[
A = \{\sigma(\infty)\in \hat A\} \quad \text{and} \quad B = \{\sigma(\infty) \in \hat B\}.
\]
For $i=1,2$ and $x \in \mathcal{V}$, write $\sigma_x^{(i)}(\infty)$ for the limit $\lim_{t \to \infty} \sigma_x^{(i)}(t)$ in the model corresponding to $(\sigma^{(i)}(0),\omega^{(i)})$ and note that these limits exist $\mathbb{P}_p$-a.s. By \eqref{eq: initial_independence}, one has
\begin{align}
&|\mathbb{P}_p(A \cap B) - \mathbb{P}_p(A) \mathbb{P}_p(B)| \nonumber \\
=~& |\mathbb{P}_p(\sigma(\infty) \in \hat A, \sigma(\infty) \in \hat B) - \mathbb{P}_p(\sigma(\infty) \in \hat A) \mathbb{P}_p(\sigma(\infty) \in \hat B)| \nonumber \\
=~& |\mathbb{P}_p(\sigma(\infty) \in \hat A, \sigma(\infty) \in \hat B) - \mathbb{P}_p(\sigma^{(1)}(\infty) \in \hat A, \sigma^{(2)}(\infty) \in \hat B)| \nonumber \\
%+~& |\mathbb{P}_p(\sigma^{(1)}(\infty)\in \hat A, \sigma^{(2)}(\infty) \in \hat B) - \mathbb{P}_p(\sigma^{(1)}(\infty) \in \hat A) \mathbb{P}_p(\sigma^{(2)}(\infty) \in \hat B)| \nonumber \\
\leq~& \mathbb{P}_p(\sigma_x(\infty) \neq \sigma_x^{(1)}(\infty) \text{ for some }x \text{ with dist}(x,o) \leq r) \label{eq: correlation_term_1} \\
+~& \mathbb{P}_p(\sigma_x(\infty) \neq \sigma_x^{(2)}(\infty) \text{ for some } x \text{ with dist}(x,o) \geq R) \label{eq: correlation_term_2}.
\end{align}

We next show that if $r \ll M \ll R$, then the terms of \eqref{eq: correlation_term_1} and \eqref{eq: correlation_term_2} are small, using both our resampling results and the finite speed of information propagation. Because these terms do not depend on $A$ or $B$, this will show Theorem~\ref{thm: correlation_decay}. We begin with \eqref{eq: correlation_term_1}. Given $\epsilon>0$ and $r$ fixed, we use the fact that all discrete spins have limits as $t \to\infty$ (see \eqref{eq: discrete_limit}) to find $T>0$ such that $\mathbb{P}(A_T) > 1-\epsilon/4$, where $A_T$ is defined by the conjunction of the following two conditions:
\begin{enumerate}
\item $\sigma_x(t) = \sigma_x(T)$ for all $t \geq T$ and all $x$ with $\text{dist}(x,o) \leq r$ and
\item $\sigma_x^{(1)}(t) = \sigma_x^{(1)}(T)$ for all $t \geq T$ and all $x$ with $\text{dist}(x,o) \leq r$.
\end{enumerate}
(Note that this event also depends on $M$, the radius of resampling, but the probability of $A_T^c$ is bounded by twice the probability that the first condition does not hold, and this is independent of $M$.) In view of the conditions defining $A_T$ for our $T$ fixed above, to guarantee that $\sigma_x(\infty) = \sigma_x^{(1)}(\infty)$ for all $x$ with $\text{dist}(x,o) \leq r$, we will show that $M \gg r$ can be chosen so that $\mathbb{P}(B_{T,M}) > 1-\epsilon/4$, where
\[
B_{T,M} = \{\sigma_x(t) = \sigma_x^{(1)}(t) \text{ for all } t < T \text{ and } x \text{ with dist}(x,o)\leq r\}.
\]
The proof is again a ``speed of propagation'' argument, similar to that given for \eqref{eq: finite_time_propagation_dominance}. If $B_{T,M}$ does not occur, then there must exist a chronological path for $[0,T]$ starting from a vertex at distance $M$ from o and ending at a vertex at distance $r$ from o. Such a path must have at least $M-r+1$ many vertices and there are $3 \cdot 2^{r-1}$ many possible ending points. So by Lemma~\ref{lem: chronological},
\[
\mathbb{P}_p(B_{T,M}^c) \leq 5e^{4T} 2^{r-1}\left( \frac{4}{5}\right)^{M-r+1}.
\]
For our fixed $r$ and $T$, then, we may choose $M$ so large that $\mathbb{P}(B_{T,M}) > 1-\epsilon/4$. On $A_T \cap B_{T,M}$, one has $\sigma_x(\infty) = \sigma_x^{(1)}(\infty)$ for all $x$ with $\text{dist}(x,o) \leq r$, so for this $M$, the probability of \eqref{eq: correlation_term_1} is bounded by $\epsilon/2$.

Moving to \eqref{eq: correlation_term_2}, we may simply apply Lemma~\ref{lem: resample_all} finitely many times (resampling the initial spin and clock at each vertex $x$ with $\text{dist}(x,o) \leq M$) to see that, a.s., 
\[
\# \Phi(\sigma^{(2)}(0),\omega^{(2)}) \Delta \Phi(\sigma(0),\omega) < \infty. 
\]
Therefore for our $M$, the probability in \eqref{eq: correlation_term_2} is at most $\epsilon/2$ for all large enough $R$. Combining this with the result of the previous paragraph, we conclude that for our given $r$, if $M$ is chosen as above and $R$ is sufficiently large, then $|\mathbb{P}_p(A\cap B) - \mathbb{P}_p(A)\mathbb{P}_p(B)|$ (from above \eqref{eq: correlation_term_1}) is at most $\epsilon/2 + \epsilon/2 = \epsilon$. This shows that $\alpha_{r,R}(p) \to 0$ as $R \to \infty$ and completes the proof.
\end{proof}

\section{Diagreement clusters: proof of Theorem~\ref{thm: disagreement_clusters}}\label{sec: disagreement_clusters}

To study disagreement clusters, we first note that item 3 of Lemma~\ref{lem: continuous_limit} states that each vertex in $\mathcal{V}$ must a.s. agree with at least one of its neighbors. Therefore there are no vertices of degree three in the disagreement graph. This means that there are only three possibilities for a component in the graph: (a) it is equal to a finite line segment, as in the claim of Theorem~\ref{thm: disagreement_clusters}, (b) it is a one-sided infinite path, so that there is a unique vertex $x$ of degree one in the component, and all other vertices have degree two, or (c) all vertices of the component have degree two. We must rule out possibilities (b) and (c), and we will do so by mass transport arguments.

Suppose that 
\begin{equation}\label{eq: type_b_assumption}
\mathbb{P}(o \text{ is in a component of type (b)})>0.
\end{equation}
Because each component of type (b) has a unique vertex of degree one, we can define the following mass transport. For $x,y \in \mathcal{V}$, set
\[
m(x,y) = \begin{cases}
1 & \quad \text{if } x \text{ is in a component of type (b) and } \\
& \qquad y \text{ is the vertex of degree one in its component} \\
0 & \quad \text{otherwise}.
\end{cases}
\]
Note that $\sum_{y\in \mathcal{V}} m(o,y) \leq 1$. However, on the event that $y$ is a vertex of degree one in a component of type (b), all vertices of its component send $y$ a unit mass. Under assumption \eqref{eq: type_b_assumption}, the root $o$ is such a vertex $y$ with positive probability, so one has $\sum_{x\in \mathcal{V}} m(x,o) = \infty$ with positive probability, and this violates the mass transport principle:
\[
1 \geq \mathbb{E}\sum_{y\in \mathcal{V}} m(o,y) = \mathbb{E} \sum_{x\in \mathcal{V}} m(x,o) = \infty.
\]
We conclude that components of type (b) cannot exist; that is, \eqref{eq: type_b_assumption} must have been false.

%We will therefore complete the proof of Theorem~\ref{thm: disagreement_clusters} if we show that the following assumption is false:
Next, if possible, assume that
\begin{equation}\label{eq: type_c_assumption}
\mathbb{P}(o \text{ is in a component of type (c)})>0.
\end{equation}
To reach a contradiction from this assumption, first apply item 3 of Lemma~\ref{lem: continuous_limit} again to see that if $o$ is in a component of type (c), since $U_o(\infty)$ equals the median of its neighboring spins, and it disagrees with two neighbors exactly, one neighbor must have a strictly higher value, and one must have a strictly lower value. Applying this to all vertices on the doubly-infinite path which constitutes the component of $o$ shows that any such component must be a strictly monotone path. In other words, it is a doubly infinite path indexed by vertices, say, $ \ldots, x_{-2},x_{-1},x_0, x_1, x_2, \ldots$ such that $U_{x_i}(\infty) < U_{x_{i+1}}(\infty)$ for all $i$. For any such component $\mathcal{C}$, we define
\[
S(\mathcal{C}) = \sup_{i \in \mathbb{Z}} U_{x_i}(\infty) \quad \text{and} \quad I(\mathcal{C}) = \inf_{i \in \mathbb{Z}} U_{x_i}(\infty).
\]
Writing $\mathcal{C}(x)$ for the component containing the vertex $x$, we can therefore under assumption \eqref{eq: type_c_assumption} find two numbers $r_1, r_2 \in [0,1]$ with $r_1 < r_2$ such that
\begin{equation}\label{eq: inf_sup_pinning}
\mathbb{P}(\mathcal{C}(o) \text{ is of type (c), } S(\mathcal{C}(o)) > r_2 \text{ and } I(\mathcal{C}(o)) < r_1)>0.
\end{equation}
Note that for any component $\mathcal{C}$ of type (c) satisfying the conditions of \eqref{eq: inf_sup_pinning} (that is, $S(\mathcal{C}) > r_2$ and $I(\mathcal{C}) < r_1$), there is exactly one vertex $p \in \mathcal{C}$ such that both $U_p(\infty) \geq \frac{r_1+r_2}{2}$ and one neighbor $z$ of $p$ has $U_z(\infty) < \frac{r_1+r_2}{2}$. Writing $p(\mathcal{C})$ for this unique vertex, we set up the mass transport as in the case of type (b) components. Namely, set
\[
m(x,y) = \begin{cases}
1 & \quad \text{if } x \text{ is in a component of type (c) and } y = p(\mathcal{C}(x)) \\
0 & \quad \text{otherwise}.
\end{cases}
\]
Once again, the total mass sent out by a vertex is at most 1, as $p(\mathcal{C})$ is uniquely determined for a component $\mathcal{C}$ of type (c); that is, $\sum_{y\in \mathcal{V}} m(o,y) \leq 1$. However on the event that $o$ is a vertex of the type $p(\mathcal{C})$ for a component $\mathcal{C}$ of type (c) satisfying $S(\mathcal{C}) > r_2$ and $I(\mathcal{C})<r_1$, the total mass entering $o$ is infinite. Due to \eqref{eq: inf_sup_pinning}, therefore, the total mass entering $o$ is infinite with positive probability, so we conclude with the same contradiction as in case (b):
\[
1 \geq \mathbb{E}\sum_{y\in \mathcal{V}} m(o,y) = \mathbb{E}\sum_{x\in \mathcal{V}} m(x,o) = \infty.
\]
This shows that components of type (c) cannot exist, and this concludes the proof of Theorem~\ref{thm: disagreement_clusters}.

\section{Howard's question: proof of Theorem~\ref{thm: spin_chains}}\label{sec: howard_question}
To prove Theorem~\ref{thm: spin_chains}, we first prove Theorem~\ref{thm: infinitely_many_ends}.
\begin{proof}[Proof of Theorem~\ref{thm: infinitely_many_ends}]
Every vertex in $\mathbb{G_+}$ has degree at least two in $\mathbb{G_+}$ by Lemma~\ref{lem: projection} and item 3 of Lemma~\ref{lem: continuous_limit}. Therefore, every component of $\mathbb{G_+}$ has at least two ends. Since $\mathbb{G_+}$ is the subgraph induced by an invariant site percolation $\{ x \in \mathcal{V}: \sigma_x(\infty) = +1 \}$ on $\mathbb{T}_3$, it follows  (see  \cite[Theorem~1.2]{Haggstrom3}) that a.s., each component of $\mathbb{G}_+$ has either exactly two ends or infinitely many ends. Note that if a component of  $\mathbb{G}_+$ has either exactly two ends, then it must be a self-avoiding doubly infinite path. Below we will rule out such a possibility and  the theorem will then follow immediately. 

Suppose, if possible, on an event $F$ with positive probability,  a component of $\mathbb{G}_+$ is a self-avoiding doubly infinite path, say $\Gamma$.
Consider the coupling between the median process and the majority dynamics with initial bias $p$ as stated in Lemma~\ref{lem: projection}.
By Theorem~\ref{thm: continuous_theorem}, the  continuous spin of each vertex only flips finitely often a.s.\  and hence, 
\[ V(\mathbb{G}_+) = \{ x \in \mathcal{V}: U_x(\infty) \le p \} \quad \text{a.s.}  \]
We claim that on the event $F$,  $U_x(\infty)   =  U_y(\infty)$ for all $x, y \in \Gamma$. If not, there exists $x \in \Gamma$ and a neighbor $y \in \partial x \cap \Gamma$ of $x$ on $\Gamma$ such that $U_x( \infty) < U_y( \infty)$. Let $z \in \partial x$ be the unique neighbor of $x$ that does not belong to the path $\Gamma$. Since $z \not \in V(\mathbb{G}_+)$, we have $U_z( \infty) > p$ and hence, $U_x( \infty) < U_z( \infty).$  This implies that with positive probability, the continuous spin of the vertex $x$ lies strictly below the spins of two of its neighbors for all large time, i.e.,  
\[ U_x(t) < \min ( U_y(t), U_z(t) )  \text{ for all large } t.\]
This is clearly a contradiction to the median update of the continuous spin at $x$ for the rings at $x$ that occurs after sufficiently  large time. So, the claim follows. Consequently, on the event $F$, the agreement cluster for any vertex in $\Gamma$ is $\Gamma$ itself and in particular, is infinite. But this is in contradiction to Theorem~\ref{thm: finite_agreement_clusters}, which says that a.s.\ all agreement clusters are finite.

\end{proof}

Given Theorem~\ref{thm: infinitely_many_ends}, we are now ready to prove Theorem~\ref{thm: spin_chains}.
\begin{proof}[Proof of Theorem~\ref{thm: spin_chains}]
Noting that the event $\{o \in \mathbb{T}^{+\text{chains}}(n)\}$ is monotone in $n$, to show Theorem~\ref{thm: spin_chains}, we must show that
\begin{equation}\label{eq: spin_chains_to_show}
\mathbb{P}_p(o \in \mathbb{T}^{+\text{chains}}(n) \text{ for all large }n) = \mathbb{P}_p(o \in \mathbb{T}^{+\text{chains}}).
\end{equation}
By symmetry, the same will hold for minus spin chains. The inequality $\leq$ is obvious, so we we need only show the inequality $\geq$, and it is trivial unless $p > p_c$. It will quickly follow from Lemma~\ref{cor: percolation}. Fix $p \in (p_c,1]$ (so that the right side of \eqref{eq: spin_chains_to_show} is positive) and consider the discrete spin model using distribution $\mathbb{P}_p$. Define an invariant bond percolation $\eta$ on $\mathbb{T}_3$ as follows. We set $\eta(e) = 1$ for an edge $e = \{x,y\}$ if $\sigma_x(\infty) = \sigma_y(\infty) = +1$, and $\eta(e) = 0$ otherwise. 
%We can equivalently define $\eta(e) = 1$ if $x$ and $y$ are in $+$ spin chains at time $\infty$, or that $e$ is an edge in the subgraph $\mathbb{T}^{+\text{chains}}$, and $\eta(e) = 0$ otherwise.
 If $\mathbb{G}_+$ denotes the subgraph of $\mathbb{T}_3$ induced by vertices with their limiting spins being equal to $+1$, then $\eta(e)=1$ if and only if $e \in E(\mathbb{G}_+)$.

We now define a  sequence of invariant bond percolations $(\eta_n)$ by setting  $\eta_n(e) = 1$ for an edge $e = \{x,y\}$ if $\sigma_x(t) = \sigma_y(t) = +1$ for all $t \geq n$ and $\eta_n(e)=0$ otherwise. Note that the hypotheses of Lemma~\ref{cor: percolation} hold for  $(\eta_n)$ and $\eta$. In other words, a.s.\ for all $e \in \mathcal{E}$, one has $\eta_n(e) \leq \eta_{n+1}(e) \leq \eta(e)$, and $\lim_{n \to \infty} \mathbb{P}_p(\eta_n(e) = 1) = \mathbb{P}_p(\eta(e) = 1)$. Therefore, by Lemma~\ref{cor: percolation}, we  conclude  that if $\mathbb{P}_p(A_\eta(o),D_\eta(o))>0$, then
\begin{equation}\label{eq: percolation_conclusion}
\lim_{n \to \infty} \mathbb{P}_p(D_{\eta_n}(o)\mid A_\eta(o),D_\eta(o)) = 1,
\end{equation}
where the events $A_{\eta}(o), D_\eta(o)$ and $D_{\eta_n}(o)$ are as defined in  Lemma~\ref{cor: percolation}.

Now we show that $\mathbb{P}(A_\eta(o),D_\eta(o))>0$, so that we can conclude \eqref{eq: percolation_conclusion}. 
For the majority dynamics, 
\[ D_\eta(o) =  \{ \sigma_o(\infty) = +1\} =  \{o \in \mathbb{T}^{+\text{chains}} \} \ \ \text{a.s.} \]
%Recall that in Section~\ref{sec: finite_agreement_p_c} (above \eqref{eq: rule_out_2_ends}) we have defined the graph $G_p$ in the continuous spin model as follows. A vertex $x$ is in the vertex set $V_p$ of $G_p$ if $U_x(t) \leq p$ for all large $t$, and an edge $e = \{x,y\}$ is in the edge set $E_p$ of $G_p$ is its endpoints are in $V_p$. In \eqref{eq: rule_out_2_ends}, we have shown that for a given $p$, a.s. no component of $G_p$ has two ends. 
In Theorem~\ref{thm: infinitely_many_ends}, we have shown that a.s.\ each nonempty component of $\mathbb{G}_+$ has infinitely many ends. 
%Since each vertex of $V_p$ has at least two neighbors in $V_p$, and the edge set $E_p$ has the same distribution as the set of $\eta$-open edges, 
Hence,  a.s., $D_\eta(o)$ occurs if and only if $A_\eta(o)$ occurs, and so
\begin{equation}\label{eq: almost_conclusion}
\mathbb{P}_p(A_\eta(o),D_\eta(o)) = \mathbb{P}_p(D_\eta(o)) = \mathbb{P}_p(\sigma_o(\infty)=+1).
\end{equation}
The right side is positive (as $p>p_c$), and so \eqref{eq: percolation_conclusion} holds.

Last, we note that $D_{\eta_n}(o)$ implies that $o \in \mathbb{T}^{+\text{chains}}(n)$, so using monotonicity of the $\eta_n$'s, \eqref{eq: percolation_conclusion}, and \eqref{eq: almost_conclusion},
\begin{align*}
\mathbb{P}(o \in \mathbb{T}^{+\text{chains}}(n) \text{ for all large }n) &\geq \mathbb{P}(D_{\eta_n}(o) \text{ occurs for all large }n) \\
&\geq \lim_{n \to \infty} \mathbb{P}(D_{\eta_n}(o) \cap A_\eta(o) \cap D_\eta(o)) \\
&= \mathbb{P}(A_\eta(o),D_\eta(o)) \\
&=\mathbb{P}(o \in \mathbb{T}^{+\text{chains}}).
\end{align*}
This shows the inequality $\geq$ in \eqref{eq: spin_chains_to_show}, and completes the proof of Theorem~\ref{thm: spin_chains}.
\end{proof}

\section{An invariant percolation result: proof of Lemma~\ref{cor: percolation}}\label{sec: percolation}
We will continue to consider the graph $\mathbb{T}_3$, but the results of this section extend to $d$-regular trees $\mathbb{T}_d$ for $d \geq 3$. The notation used in this section will be that which was introduced in Section~\ref{sec: tools}.

Our main result shows that if the density of 1's in a sequence $(\eta_n)$ of invariant bond sub-percolation processes of another invariant bond percolation process $\eta$ converges to that of $\eta$, and if $\eta$ has open clusters with at least three ends with positive probability, then for large $n$, $\eta_n$ has infinite open clusters with positive probability. This result will imply Lemma~\ref{cor: percolation}, which we will prove afterward.
\begin{prop}\label{prop: percolation}
Let $\eta$ be an invariant bond percolation process on $\mathbb{T}_3$ and let $(\eta_n)$ be a sequence of invariant bond percolation processes on $\mathbb{T}_3$ such that 
\begin{enumerate}
\item $\eta_n \leq \eta$ a.s. for all $n$, and
\item for any $e \in \mathcal{E}$, $\lim_{n\to\infty} \mathbb{P}(\eta_n(e) = 1) = \mathbb{P}(\eta(e)=1)$.
\end{enumerate}
Then if $\mathbb{P}(A_\eta(o))>0$, one has
\[
\lim_{n\to\infty} \mathbb{P}(\#C_{\eta_n}(o) = \infty \mid A_\eta(o)) = 1.
\]
\end{prop}

\begin{rem}
For $\eta \equiv 1$, this result reduces to that of H\"aggstr\"om \cite[Theorem~1.6]{Haggstrom3}, which guarantees that invariant percolation on $\mathbb{T}_3$ with high enough density of open edges must contain an infinite open cluster with positive probability.
\end{rem}

\begin{proof}
The proof is a modification of the argument from \cite[Section~4]{BLPS}. We first remark that a.s., if an open component of $\eta$ has at least three ends, it has infinitely many ends, and furthermore has no isolated ends. One way to say this  on $\mathbb{T}_3$ is in terms of triple points. 
%We say that the vertex $v$ is a triple point for the process $\eta$ if $C_\eta(v)$ is infinite, and the removal of $v$ and its incident edges from $C_\eta(v)$ splits it into at least three infinite components. 
Note that an open component has at least three ends if and only if it has a triple point, and it has infinitely many ends if and only if it has infinitely many triple points. Therefore, under the assumption that $A_\eta(o)$ has positive probability, $o$ has positive probability to be a triple point for $\eta$, since  ``everything shows up at the root'' (see \cite[Lemma~2.3]{aldous_lyons}). We may then state the isolated ends property as follows: a.s., if $\Gamma$ is an infinite self-avoiding path contained in an open component of $\eta$, then $\Gamma$ must contain a triple point (and therefore infinitely many). This is a standard fact that follows, for instance, from \cite[Lemma~4.2]{BLPS}.

So given an outcome, we construct the {\em triple-point graph} $T$ of $\eta$ as follows. The vertex set $V(T)$ of $T$ is the set of $v \in \mathcal{V}$ such that $v$ is a triple point of $\eta$, and the edge set $E(T)$ is the set of $\{v,w\}$ with $v \neq w$ such that $v,w \in V(T)$, $w \in C_\eta(v)$, and there are no other triple points of $\eta$ on the unique (vertex) self-avoiding path of $\mathbb{T}_3$ from $v$ to $w$. (That is, $w$ is the closest triple point to $v$ on the infinite ray of $\mathbb{T}_3$ starting from $v$ and containing $w$.) Because a.s., each open component of $\eta$ that contains a triple point has no isolated ends, the triple point graph $T$ is a.s.\ a disjoint union of $3$-regular trees.

We now define a sequence of  bond percolation processes on the triple-point graph $T$ constructed  from $\eta_n$. Let us say that an edge $\{v,w\}$ of $E(T)$ is $n$-open if all edges on the unique (vertex) self-avoiding path from $v$ to $w$ in the original tree $\mathbb{T}_3$ are open in $\eta_n$. Otherwise, it is called $n$-closed. Given $v \in V(T)$, we can similarly define the $n$-cluster of $v$, written $\mathsf{C}_n(v)$, as the subgraph of $T$ whose vertices are the $w \in V(T)$ such that there is an $n$-open path from $w$ to $v$ in $T$, and whose edges are those $\{w,z\}$ which are $n$-open and whose endpoints are in the vertex set $V(\mathsf{C}_n(v))$ of $\mathsf{C}_n(v)$. We orient the edges of the finite $n$-clusters toward ``random centers'' as follows. 
%: (a) finite $n$-clusters, (b) infinite $n$-clusters with 1 end, (c) infinite $n$-clusters with 2 ends and (d) infinite $n$-clusters with at least 3 ends. (The number of ends is being measured in the triple point graph $T$.) We orient some of the edges of $T$ based on these four cases. {\color{red} put in the intuitive definition first before the real ones below.}
%\begin{enumerate}
%\item[(a) ]If $\{v,w\}$ is an edge of an $n$-cluster $C$ that is finite, then we orient the edge ``toward a random center'' in $C$.
We let $(\xi_v)_{v \in \mathcal{V}}$ be a family of i.i.d.\ uniform $[0,1]$ random variables associated to the vertices of the original graph $\mathbb{T}_3$, and for a finite $n$-cluster $\mathsf{C}$, we select $v_{\mathsf{C}}$ to be the vertex of $\mathsf{C}$ in $T$ with the largest value of $\xi$. Then every edge of finite $n$-cluster $\mathsf{C}$ is oriented ``towards" $v_{\mathsf{C}}$. In other words, an edge $\{v,w\}$ in this $n$-cluster is oriented as $(v,w)$, where $w$ is the unique vertex of $\{v,w\}$ such that there is a self-avoiding path in $\mathsf{C}$ from $w$ to $v_{\mathsf{C}}$ that does not contain $v$.
%\item[(b) ]If $\{v,w\}$ is an edge in an $n$-cluster of $T$ which has 1 end, we orient this edge in the order $(v,w)$, where $w$ is the unique element of $\{v,w\}$ from which there is an infinite (vertex) self-avoiding path in $C_n(v)$ that does not touch $v$. (Note that there is only one such vertex of $v,w$ with this property, or else $C_n(v)$ would have at least 2 ends.) In other words, all edges of such a component $C_n(v)$ are oriented ``toward infinity.'' 
%\item[(c) ]For any $n$-cluster $C$ of $T$ with 2 ends, there must be a unique doubly-infinite self-avoiding path $P(C)$ contained in $C$. This is because $C$ is a 2-ended subgraph of a 3-regular tree with no triple points. For any edge $\{v,w\}$ with both $v,w \in P(C)$, we leave the edge $\{v,w\}$ unoriented. For other edges of $C$, we orient them toward $P(C)$ as follows. Removing all vertices of $P(C)$ and their incident edges splits $C$ into only finite components, and each of these components $\hat C$ has a unique vertex $v_{\hat C}$ which is adjacent to $P(C)$ in $C$. The edge $\{v_{\hat C},w\}$, where $w \in P(C)$, is oriented as $(v_{\hat C},w)$, and every other edge $\{v,w\}$ for $v,w \in \hat C$ is oriented as $(v,w)$, where $w$ is the unique element of $\{v,w\}$ such that there is a (vertex) self-avoiding path in $C$ from $w$ to $P(C)$ that does not contain $v$.
%\item[(d) ] If $\{v,w\}$ is an edge of an $n$-cluster of $T$ which has at least 3 ends, we leave it unoriented.
%\end{enumerate}
The orientations depend on the process $\eta_n$, so we will refer to them as $n$-orientations in the edge set $E(T)$. Now if $v \in V(T)$ is a vertex with a finite $n$-cluster, we define the backward $n$-cluster of $v$ to be the directed graph whose vertices are those $w$ such that there is an $n$-oriented path from $w$ to $v$ in $V(T)$ and whose edge set consists of those edges which are $n$-oriented and whose endpoints are vertices of the backward $n$-cluster of $v$. By this definition, each $v \in V(T)$ with a finite $n$-cluster is a vertex of its own (finite) backward $n$-cluster. For such $n$, we may then define the set of (external) ``backward $n$-leaves'' of $v$ to be the vertices $w \in V(T)$ such that there is an edge $\{w,z\} \in E(T)$ which is $n$-closed and $z$ is in the backward $n$-cluster of $v$.

Based on these $n$-orientations, we define a mass transport $m_n$ on $\mathbb{T}_3$. If $v \in V(T)$ has a finite $n$-cluster, we send a unit flow (the $n$-flow) toward the backward $n$-leaves of $v$. That is, $v$ sends the flow backward along oriented edges, and when it reaches a vertex $w$ in the backward $n$-cluster of $v$, the flow splits evenly among those vertices $z \in V(T)$ such that either $z$ is a backward $n$-leaf of $v$ or that $\{z,w\}$ is $n$-oriented from $z$ to $w$. In this way, the total flow received at $w$ from $v$ is at most $2^{-l(v,w)}$, where $l(v,w)$ is the number of edges in the unique self-avoiding path from $v$ to $w$ in $T$. Now the mass transport $m_n(x,y)$ is defined as the magnitude of the $n$-flow received at $y$ from $x$ if $x,y \in V(T)$, the $n$-cluster of $x$ is finite, and $y$ is a backward $n$-leaf for $x$. 

Note that the total mass sent out by an $x \in \mathcal{V}$ is 1 if $x$ is in $V(T)$ and in a finite $n$-cluster, and 0 otherwise. 
%Indeed, if $x$ is not in a finite $n$-cluster, the definition of $m_n(x,y)$ entails that this quantity is zero. Otherwise, the backward $n$-cluster of $x$ is an oriented tree rooted at $x$, and so the total flow received at the leaves is equal to the total flow sent out from the root, which is 1. 
Therefore for any $n$,
\[
\mathbb{E}\sum_{y\in \mathcal{V}} m_n(x,y) = \mathbb{P}(x \in V(T),~\#\mathsf{C}_n(x) < \infty).
\]
The total mass received by a vertex $x$ is only nonzero if it is in $V(T)$, and is a backward $n$-leaf of another vertex in $V(T)$ whose $n$-cluster is finite. In this case, $x$ can receive mass from vertices in at most three $n$-clusters (those of its neighbors), and only if $x$ is a backward $n$-leaf of a vertex in that $n$-cluster. Furthermore, an edge incident to $x$ in $T$ must be $n$-closed. In each of these $n$-clusters corresponding to a neighbor $z$, $x$ therefore receives mass only from vertices in the forward oriented path starting at $z$ (and not from itself). Thus the expected mass received is at most
\begin{align*}
\mathbb{E}\sum_{y\in \mathcal{V}} m_n(y,x) &\leq \mathbb{E}\left[ 3 \cdot \sum_{k=1}^\infty 2^{-k}  \mathbf{1}_{\{x \in V(T),~x \text{ has an incident }n\text{-closed edge}\}} \right] \\
&= 3\mathbb{P}(x \in V(T),~x \text{ has an incident }n\text{-closed edge}).
\end{align*}
By the mass transport principle, the expected mass in equals the expected mass out, so
\begin{align}
&\mathbb{P}(\#\mathsf{C}_n(x)<\infty \mid x \in V(T)) \nonumber \\
\leq ~& 3\mathbb{P}(x \text{ has an incident }n \text{-closed edge} \mid x \in V(T)). \label{eq: tacos_1}
\end{align}
Recalling that $\text{dist}$ denotes the graph distance on $\mathbb{T}_3$ and letting $E_M(n;x)$ be the event that $\eta(e) \neq \eta_n(e)$ for some $e \in \mathcal{E}$ with $\text{dist}(e,x) \leq M$ (here $\text{dist}(e,x)$ is the minimal distance from $x$ to an endpoint of $e$), then
\begin{align}
&\mathbb{P}(x \text{ has an incident } n\text{-closed edge} \mid x \in V(T)) \nonumber \\
\leq ~& \mathbb{P}(x \text{ has a neighbor } z \text{ in }V(T) \text{ with } \text{dist}(x,z) \geq M \mid x \in V(T)) \label{eq: tacos_2}\\
+~&\mathbb{P}(E_M(n;x) \mid x \in V(T)). \label{eq: tacos_3}
\end{align}
The probability in \eqref{eq: tacos_2} can be made small by choosing $M$ large. For any given $M$, the probability in \eqref{eq: tacos_3} converges to 0 as $n \to \infty$ by item 1 in the statement of Proposition~\ref{prop: percolation}. Therefore \eqref{eq: tacos_1} converges to 0 as $n \to \infty$; that is,
\begin{equation}\label{eq: two_ends}
\lim_{n\to\infty} \mathbb{P}(\#\mathsf{C}_n(x)<\infty \mid x \in V(T)) = 0.
\end{equation}

Now to prove Proposition~\ref{prop: percolation}, we will show that
\[
\lim_{n\to\infty} \mathbb{P}(A_\eta(o), \#C_{\eta_n}(o)<\infty) = 0.
\]
We decompose this probability for $M>0$ as
\begin{align}
&\mathbb{P}(A_\eta(o), \#C_{\eta_n}(o) < \infty) \\
\leq~&\mathbb{P}(A_\eta(o),~\text{there is no }x \in V(T) \cap C_\eta(o) \text{ with } \text{dist}(o,x) \leq M) \label{eq: tacos_4}\\
+~&\mathbb{P}(E_M(n;o)) \label{eq: tacos_5}\\
+~&\sum_{x \in \mathcal{V} : \text{dist}(o,x) \leq M} \mathbb{P}(A_\eta(o),\#C_{\eta_n}(o) < \infty, x \in V(T) \cap C_\eta(o),E_M(n;o)^c). \label{eq: tacos_6}
\end{align}
Again, if $M$ is large, the probability in \eqref{eq: tacos_4} can be made small, and for fixed $M>0$, the probability in \eqref{eq: tacos_5} converges to 0 as $n \to\infty$. So we must only show that for fixed $M>0$ and $x \in \mathcal{V}$ with $\text{dist}(o,x)\leq M$, the summand in \eqref{eq: tacos_6} corresponding to $x$ converges to 0 as $n\to\infty$. However, if $x \in V(T) \cap C_\eta(o)$ has $\text{dist}(o,x)\leq M$ and $E_M(n;o)^c$ occurs, then all edges in $C_\eta(o)$ connecting $o$ to $x$ are open in $\eta_n$. This means that $x \in V(T) \cap C_{\eta_n}(o)$ and thus $C_{\eta_n}(x) = C_{\eta_n}(o)$. If, in addition, $\#C_{\eta_n}(o) < \infty$, then $\#\mathsf{C}_n(x) < \infty$. Therefore
\[
\mathbb{P}(A_\eta(o), \#C_{\eta_n}(o) <  \infty,x \in V(T) \cap C_\eta(o),E_M(n;o)^c) \leq \mathbb{P}(\#\mathsf{C}_n(x) < \infty, x \in V(T)).
\]
But by \eqref{eq: two_ends}, this probability converges to 0, and this completes the proof.
\end{proof}

We are now ready to prove Lemma~\ref{cor: percolation} from Proposition~\ref{prop: percolation}. Note that its assumptions are stronger than those in Proposition~\ref{prop: percolation}: the $\eta_n$'s are here assumed to be monotone. 
%Let $D_\eta(v)$ be the event that $v$ is in a self-avoiding doubly-infinite $\eta$-open path. 
%\begin{cor}\label{cor: percolation}
%Let $\eta$ be an invariant bond percolation process on $\mathbb{T}_3$ and let $(\eta_n)$ be a sequence of invariant bond percolation processes on $\mathbb{T}_3$ such that 
%\begin{enumerate}
%\item $\eta_n \leq \eta_{n+1} \leq \eta$ a.s. for all $n$, and
%\item for any $e \in \mathcal{E}$, $\lim_{n\to\infty} \mathbb{P}(\eta_n(e) = 1) = \mathbb{P}(\eta(e)=1)$.
%\end{enumerate}
%Then if $\mathbb{P}(A_\eta(o),D_\eta(o))>0$, one has
%\begin{equation}\label{eq: to_prove_doubly_infinite}
%\lim_{n\to\infty} \mathbb{P}(D_{\eta_n}(o) \mid A_\eta(o), D_\eta(o)) = 1.
%\end{equation}
%\end{cor}
\begin{proof}[Proof of Lemma~\ref{cor: percolation}]
By Proposition~\ref{prop: percolation}, one has $\lim_{n \to \infty} \mathbb{P}(\#C_{\eta_n}(o) = \infty \mid A_\eta(o)) = 1$, and therefore using monotonicity of the $\eta_n$'s,
\[
\mathbb{P}(\#C_{\eta_n}(o) = \infty \text{ for all large }n \mid A_\eta(o),D_\eta(o)) = \lim_{n \to \infty} \mathbb{P}(\#C_{\eta_n}(o) = \infty \mid A_\eta(o),D_\eta(o)) = 1.
\]
%So it suffices to show that
%\[
%\limsup_{n \to \infty} \left[ \mathbb{P}(A_{\eta_n}(o) \mid A_\eta(o),D_\eta(o)) - \mathbb{P}(D_{\eta_n}(o) \mid A_\eta(o),D_\eta(o)) \right] \leq 0.
%\]
%Because the $\eta_n$'s are monotone, this is equivalent to showing
%\begin{align*}
%&\mathbb{P}(A_{\eta_n}(o) \text{ occurs for all large }n \mid A_\eta(o),D_\eta(o)) \nonumber \\
%\leq~& \mathbb{P}(D_{\eta_n}(o) \text{ occurs for all large }n \mid A_\eta(o), D_\eta(o)). 
%\end{align*}
From this and monotonicity again, it suffices to show that
\begin{equation}\label{eq: corollary_to_show}
\mathbb{P}(\#C_{\eta_n}(o) = \infty \text{ for all large }n \text{ but } D_{\eta_n}(o)^c \text{ occurs for all }n \mid A_\eta(o), D_\eta(o)) = 0.
\end{equation}

So suppose that the event in \eqref{eq: corollary_to_show} occurs. Then if $n$ is large, $\#C_{\eta_n}(o) = \infty$, so there is an infinite self-avoiding open path $\Gamma$ starting from $o$ in $\eta_n$, but it cannot be extended in any ``backward'' direction because $D_{\eta_n}(o)$ does not occur. That is, if the first vertex of $\Gamma$ other than $o$ is $v$, then there is no self-avoiding infinite open path in $\eta_n$ starting from $o$ whose first vertex after $o$ is not $v$. If we increase $n$, then $\Gamma$ will remain an open path in $\eta_n$ (this is simply because the $\eta_n$'s are increasing), but we will show that another such infinite path will emerge that has its first vertex after $o$ not equal to $v$. These two paths will form a doubly infinite path in $\eta_n$ containing $o$, and will force $D_{\eta_n}(o)$ to occur. We will then conclude \eqref{eq: corollary_to_show}.

To formalize this argument, define, for any neighbor $v$ of $o$, the set $T_v$ of vertices $z$ such that there is a self-avoiding $\eta$-open path from $z$ to $o$ that contains $v$. In words, $T_v$ is the set of vertices in the ``subtree' of $o$ in $\eta$ beginning with $v$. Consider the event $A$ comprised of the following conditions:
\begin{enumerate}
\item $\#C_{\eta_n}(o)=\infty$ for all large $n$, and
\item for all large $n$, and for each neighbor $v$ of $o$ such that $T_v$ is infinite, there are infinitely many) vertices $z\in T_v$ for which $\#C_{\eta_n}(z) = \infty$.
\end{enumerate}
We will show that
\begin{equation}\label{eq: mass_transport_howard}
\mathbb{P}(A \mid A_\eta(o) \cap D_\eta(o) \cap \{\#C_{\eta_n}(o)=\infty\} \text{ occurs for all large }n) = 1.
\end{equation}
First, let us see why \eqref{eq: mass_transport_howard} implies \eqref{eq: corollary_to_show} and therefore completes the proof of Lemma~\ref{cor: percolation}. Assume that $A \cap A_\eta(o) \cap D_\eta(o) $ occurs but that for all $n$, $D_{\eta_n}(o)$ does not occur. As we said above, if $n$ is large enough so that $\#C_{\eta_n}(o)=\infty$, then $o$ is in an infinite self-avoiding open path $\Gamma$ in $\eta_n$ which visits a neighbor of $o$, say $v$, directly after leaving $o$. Because $D_\eta(o)$ occurs, there is another neighbor $w \neq v$ of $o$ such that $T_w$ is infinite. Because $D_{\eta_n}(o)$ does not occur, the set of vertices in $T_w$ which are connected by $\eta_n$-open edges to $o$ must be finite. Hence,  by condition 2 above, we can find a vertex $z$ in this $T_w$ that is not connected  by $\eta_n$-open edges to $o$ and for which $\#C_{\eta_n}(z)=\infty$. Note that $z$ that is not connected by $\eta_n$-open edges to $o$ and $z$ must be in an infinite $\eta_n$-open path $P_z$ that does not contain $o$. Letting $P$ be the path connecting $z$ to $o$ in $\mathbb{T}_3$, then $P$ is $\eta$-open, and by condition 3 above, all edges in $P$ will, for even larger $n$, be $\eta_n$-open. Since the configurations $\eta_n$ are nondecreasing, this implies that for large $n$, all edges in $P \cup P_z \cup \Gamma$ are open in $\eta_n$. From this union we can extract a doubly-infinite self-avoiding path in $\eta_n$ containing $o$ (since $\Gamma$ does not contain $z$ and $P_z$ does not contain $o$), and this contradicts the fact that for all $n$, $D_{\eta_n}(o)^c$ occurs. Therefore
\[
\{A \cap A_\eta(o) \cap D_\eta(o) \} \subseteq \{D_{\eta_n}(o) \text{ occurs for all large }n\}.
\]
Using \eqref{eq: mass_transport_howard}, the probability in \eqref{eq: corollary_to_show} is bounded above by
\[
\mathbb{P}(A^c \cap \{\#C_{\eta_n}=\infty\} \text{ occurs for all large }n \mid A_\eta(o), D_\eta(o)) = 0.
\]

We are left then to prove equation \eqref{eq: mass_transport_howard}, and this follows from a mass transport argument. Because the $\eta_n$'s are nondecreasing and $\lim_{n \to \infty} \mathbb{P}(\eta_n(e) = 1) = \mathbb{P}(\eta(e)=1)$, a.s., for any fixed $e \in \mathcal{E}$, $\eta_n(e) = \eta(e)$  for all large $n$. So we will show that if we define the event $E_n$ that either $\#C_{\eta_n}(o)<\infty$ or that both $\#C_{\eta_n}(o)=\infty$ and for each neighbor $v$ of $o$ such that $T_v$ is infinite, there are infinitely many vertices $z \in T_v$ such that $\#C_{\eta_n}(z)=\infty$, then
\begin{equation}\label{eq: condition_two_howard}
\mathbb{P}(E_n) = 1 \quad \text{for all }n.
\end{equation}
For a contradiction, suppose that the probability of \eqref{eq: condition_two_howard} is $<1$ for some $n$. % so that its complement has positive probability. 
We now define a mass transport as follows. Set
\[
m(x,y) = \begin{cases}
1 & \quad \text{if } y \in C_\eta(x)\text{ and } y \text{ is the closest vertex to } x\text{ such that } \#C_{\eta_n}(y)=\infty\\
0 & \quad \text{otherwise}
\end{cases}.
\]
(Here we use an invariant tie-breaking rule, as in Section~\ref{sec: tools}.) We will use the convention that if $\#C_{\eta_n}(x)=\infty$, then $m(x,x) = 1$. We have assumed that with positive probability, $\#C_{\eta_n}(o)=\infty$, but that there is a neighbor $v$ of $o$ such that $T_v$ is infinite but for which there are only finitely many vertices $z \in T_v$ such that $\#C_{\eta_n}(z)=\infty$. On this event, there is a vertex $w$ which receives infinite total mass. Because $o$ has positive probability to be such a vertex $w$ under our assumption that \eqref{eq: condition_two_howard} is false, we find that $\mathbb{E}\sum_{x\in \mathcal{V}} m(x,o) = \infty$. However $\sum_{x\in \mathcal{V}} m(o, x) \le 1$, and this contradicts the mass transport principle. We conclude that \eqref{eq: condition_two_howard} must be true, and therefore \eqref{eq: mass_transport_howard} holds. This completes the proof of Lemma~\ref{cor: percolation}.
\end{proof}

%\bigskip
%\noindent
%{\bf Acknowledgements.} {\color{red} anything to add?}

\section{Open problems}

Here we list some future questions.

\begin{enumerate}

\item Is $p \mapsto \theta(p)$ differentiable, especially at $p=p_c$? 

\item \cite[Question~13.6]{P19} Is $p_c = 1/2$ on the $d$-regular tree for $d \geq 4$? When $d$ is even, ties are broken typically using a fair coin, and we can still define
\[
p_c = \sup\{ p \in [0,1] : \mathbb{P}_p(\sigma_o(\infty) \text{ exists and equals } -1) = 1\}.
\]
For $d=3,5,7$, with synchronous updates (at integer times), Kanoria-Montanari \cite{KM11} showed that $p_c<1/2$.

\item For $d \geq 5$ and $d$ odd, is $\theta(p)$ a continuous function of $p$, especially at $p=p_c$?

\item Extend Theorem~\ref{thm: correlation_decay} to all $p \in [0,1]$ (instead of Lebesgue-a.e. $p$).

\item Can one prove an exponential rate of correlation decay in Theorem~\ref{thm: correlation_decay}?

\item From Theorem~\ref{thm: finite_agreement_clusters}, all agreement clusters in the median process are finite. What is the rate of decay of
\[
\mathbb{P}(\text{number of vertices in } \mathcal{A}_o \geq \lambda) \text{ as } \lambda \to \infty?
\]

\item Let $T$ be the last time that the spin at the root in the median process flips. What is the rate of decay of
\[
\mathbb{P}(T \geq \lambda) \text{ as } \lambda \to \infty?
\]
This question is also open for the discrete model. (See \cite[Sec.~3]{HowardTree}, where exponential decay is shown for $p$ close enough to 0 or 1.)

\end{enumerate}

\section*{Appendix}

First we prove Lemma~\ref{lem: chronological}, which bounds the probability of existence of long chronological paths.
\begin{proof}[Proof of Lemma~\ref{lem: chronological}]
If $\Gamma$ is a deterministic path starting from $o$ with $\ell$ many vertices, the time it takes for successive clock rings to occur along $\Gamma$ is at least $\sum_{i=1}^\ell \tau_i$, where the $\tau_i$'s are i.i.d.~exponential random variables with mean one. There are $4^{\ell-1}$ many paths starting from or ending at $o$ with $\ell$ many vertices, so by a union bound and the Markov inequality,
\begin{align*}
&\mathbb{P}(\exists \text{ chronological path with } \ell \text{ many vertices starting from or ending at }o \text{ for } [0,T]) \\
\leq~&4^{\ell-1} \mathbb{P}\left( \sum_{i=1}^\ell \tau_i \leq T\right) \\
=~& 4^{\ell-1} \mathbb{P}\left( \exp\left( - 4 \sum_{i=1}^\ell \tau_i \right) \geq e^{-4T}\right) \\
\leq~&4^{\ell-1} \frac{\mathbb{E}\exp\left( - 4 \sum_{i=1}^\ell \tau_i\right)}{e^{-4T}} 
= \frac{e^{4T}}{4} \left( \frac{4}{5} \right)^\ell.
\end{align*}
Summing this bound over $\ell \geq k$ gives the statement of the lemma.
\end{proof}

Next, we state and prove a result showing that if $q \in [0,1]$ has $\theta(q)>0$, then with positive probability, in the majority vote model with initial bias $q$, the root starts with spin $+1$ and never flips.

\begin{lem}\label{lem:never_flip}
Consider the majority dynamics on $\mathbb{T}_3$ with  initial spin configuration distributed according to the i.i.d.~product measure $\mu_q$ with $q \in [0,1]$  satisfying  $\theta(q) > 0$. Let us keep the spin one of the neighbors, say $x_0$, of the root $o$  frozen at $-1$ for all time. Then with positive probability, the spin at the root is $+1$ at $t=0$ and it never flips. Clearly, this event  depends only on the clocks and initial spins of the vertices in the subtree $\mathbb{T}_{o \to x_0}$. 
\end{lem}

\begin{proof}
Let us denote the three neighbors of $o$ in $\mathbb{T}_3$ by $x_{-1}, x_0, x_1$. 
By a result of Harris \cite{Harris, liggett}, it follows that 
the measure  $\mu^t$ on $\{-1, 1\}^{ \mathcal{V}}$ describing the state $\sigma(t)$ of the system at time $t \in [0, \infty]$ possesses the FKG property; i.e., increasing functions of the spin variables are positively correlated. In fact, this follows from   the FKG property  of $\mu^0$ (which holds trivially since $\mu^0$  is an i.i.d.\ measure) and the attractiveness of the Markov process. Therefore, 
\[ \prob_q( \sigma_{x_{-1}}(\infty) = +1,  \sigma_{x_{1}}(\infty) = +1 ) \ge \prob_q( \sigma_{x_{-1}}(\infty) = +1) \prob_q( \sigma_{x_{1}}(\infty) = +1 )  = \theta^2(q)> 0.\]
Consequently, for some large fixed time $T$, 
 \[ \text{ the event }   A:= \{ \sigma_{x_{-1}}(t) = +1,  \sigma_{x_{1}}(t) = +1  \ \text{ for all } t \ge T \}  \  \text{ has positive probability}.\]
As before, let $\sigma(0)$ be the discrete spins at $t=0$ and $\omega$ be a realization of the Poisson clocks of the vertices in $\mathbb{T}_3$.   We define a modification operator $\Psi: (\sigma(0), \omega) \mapsto (\sigma'(0), \omega')$  by setting the initial spin at $o$ to be $+1$ and by suppressing all clock rings of $o$ in $[0, T]$ so that the first ring of the clock at $o$ happens after time $T$.  Let $A'$ be the event obtained from $A$ after applying this modification, i.e., $A' = \{ \Psi((\sigma(0), \omega) ): (\sigma(0), \omega) \in A\}$. Then $A'$ also has positive probability.

Since $1= \sigma'_o(t) \ge \sigma_o(t)$ for all $0 \le t \le T$,  we claim that $\sigma'_y(t) \ge \sigma_y(t)$ for all $0 \le t \le T$ and for each vertex $y$. Indeed, 
at the time of each clock ring at any $x \in \partial o$  in the interval $[0, T]$, the spin of  its neighbor $o$ in $(\sigma'(0), \omega')$ dominates that in $(\sigma(0), \omega)$. Therefore, 
we have $\sigma'_{x}(t)  \ge \sigma_{x}(t)$  for all $0 \le t \le T$ and for any $x \in \partial o$. Applying the argument iteratively to the vertices lying at distance $r=1,2, \ldots$ from $o$ yields the claim. 

Since $\sigma'_y(T) \ge \sigma_y(T)$ for all $y$ and the clock rings at every vertex are identical in $\omega$ and $\omega'$ after $T$, it follows from the attractiveness property of the majority dynamics that 
$\sigma'_x(t) \ge \sigma_x(t)$ for all $x$ and for all $t \ge T$. In particular,  $\sigma'_{x_{-1}}(t) = +1$ and  $\sigma'_{x_{1}}(t) = +1$ for all $ t \ge T$. Therefore, at the time of the each ring in $\omega_o'$ (which, by definition,  occurs after time $T$), the vertex $o$ has at least two neighbors with $+1$ spins. Hence, $\sigma'_o(t) =+1$ for all time $t$ after $T$. So, on the event $A'$, the spin at the root is $+1$ at $t=0$ and it never flips. The lemma follows. 
\end{proof}

\bigskip
\noindent
{\bf Acknowledgements.} AS thanks M. Bramson, E. Mossel, and O. Tamuz for helpful discussions.

\end{document}